\documentclass[red,11pt,a4paper]{article}
\usepackage{hyperref}
\usepackage{multicol} 
\usepackage{graphics,epsfig,color}
\usepackage{amsmath,amsthm,amssymb}
\usepackage{framed}
\usepackage{amsfonts}
\usepackage{epstopdf}             
\usepackage{tikz}
\usetikzlibrary{positioning,decorations,arrows,patterns,shapes,backgrounds}

\title{Figure's annotation}

\newtheorem{theorem}{Theorem}[section]
\newtheorem{remark}{Remark}[section]

\newtheorem{definition}{Definition}[section]
\newtheorem{lemma}[theorem]{Lemma}

\newtheorem {proposition}[theorem]{Proposition}

\newcommand{\RR}{\mathbb{R}}

\newcommand{\Dt}{\dfrac{d}{dt}}
\newcommand{\sign}{sign}
\newcommand{\intO}{\int_{\Omega}}
\newcommand{\intTO}{\int_{0}^{T}\int_{\Omega}}
\newcommand{\intT}{\int_{0}^{T}}

\DeclareMathOperator{\coker}{Coker}

\DeclareMathOperator{\curl}{curl}

\DeclareMathOperator{\divv}{div}
\newcommand{\Ss}{\mathbf{S}}

\newcommand{\vrho}{\varrho}
\newcommand{\go}{\rightarrow}
\newcommand{\uu}{\mathbf{u}}

\DeclareMathOperator{\D}{D}

\newcommand{\del}{\partial}
\def\Xint#1{\mathchoice
{\XXint\displaystyle\textstyle{#1}}%
{\XXint\textstyle\scriptstyle{#1}}%
{\XXint\scriptstyle\scriptscriptstyle{#1}}%
{\XXint\scriptscriptstyle%
\scriptscriptstyle{#1}}%
\!\int}
\def\XXint#1#2#3{{\setbox0=\hbox{$#1{#2#3}{%
\int}$ }
\vcenter{\hbox{$#2#3$ }}\kern-.6\wd0}}

\def\dashint{\Xint-}

\def\edc{\end{document}}
\numberwithin{equation}{section}

\date{\today}

\title{Well-posedness for a class of compressible non-Newtonian fluids  equations}
\author{Bilal Al Taki$^{1,2}$}
\begin{document}
\maketitle


\normalsize
\begin{center}
{\small 1. Sorbonne Université, CNRS, Laboratoire Jacques-Louis Lions, 75005 Paris, France. 
\\
2. Peking University, Beijing International Center for Mathematical Research, 100871 Beijing,  China.\\
email: bilal.al\_taki@sorbonne-universite.fr}
\end{center}

\begin{abstract}
The purpose of this paper is to deal with the issue of well-posedness for a class of non-Newtonian fluid dynamics equations. These sets of equations are commonly used to describe various complex models that appear in nature, industry, and biology. 
The equations describing the motion of such fluids are characterized by a non-linear constitutive law relating the state of stress to the rate of deformation. We show the local-in-time existence and uniqueness of strong solutions to two important models: the Power Law model and the Bingham model. While our result for the first model holds over a periodic domain $\Omega=\mathbb{R}^3,$ the result obtained on the second model is limited to the one-dimensional case. This is because Bingham's constitutive law is discontinuous due to phase transition that may appear during the time when flows change nature, particularly from liquid motion to rigid motion and vice-versa. This property reduces the probability of showing smooth solutions to such a system in higher dimension space. 
\end{abstract}

\maketitle


\section{Introduction}
Non-Newtonian fluids are characterized by different features, such as viscosity, elasticity, or memory effects. These fluids exhibit either shear thinning or shear thickening behavior and in some cases, exhibit a yield stress, i.e., a stress level which must be overcome before the fluid begins to flow. There are a number of common rheology models or methods often used to characterize non-Newtonian fluids. These include the Herschel-Bulkley, Power Law, and Bingham Plastic models.
 At the continuum level, the mathematical model describing the motion of fluids when changes in temperature are not taken into account is usually given by the following system
\begin{equation}\label{depart}
\begin{array}{ccc}
\del_{t}\rho+\divv(\rho \uu)=0\vspace*{0.2cm}\\
\del_{t}(\rho \uu)+\divv(\rho \uu\otimes \uu)-\divv \Ss +\nabla p=\rho f,
\end{array}
\end{equation}
where $\uu$ is the velocity, $\rho$ is the density, $p$ is the pressure which is assumed here to be a given function of $\rho$, i.e. $p=a\rho^\gamma$ where $a>0, \gamma>1$. $\Ss=(\Ss_{ij})$ represents the stress tensor, $f$ is the external body force. In order to characterize the specific fluid under consideration we need the so-called constitutive law which relate the stress tensor $\Ss$ with the shear rate $\D(\uu)=\frac{1}{2}(\nabla \uu+\nabla^t \uu)$. Because of their variety (elasticity, plasticity, normal stresses, memory, etc), we cannot expect that a single expression for $\Ss$ could be exist, describing all these types of fluids. Therefore, different models were developed in the literature to describe various effects. At this stage, the reader is referred to \cite{Ma-Ra-law} where the authors investigated the development of constitutive relations for compressible fluids within the thermodynamic framework.

\medskip

 Accordingly, the modeling and analysis of non-Newtonian fluids remain a mathematical and numerical challenge in a variety of industrial and natural applications. This is greatly due to the complex physical phenomena involved and the computational cost associated with the simulations of such flows. As a consequence, few results have been shown on the mathematical analysis of non-Newtonian fluids compared to Newtonian ones. Furthermore, we often consider fluids of Newtonian type instead of non-Newtonian one when studying the issue of well-posedness of some mathematical models arising from fluid mechanics, and that is for the sake of simplicity of analysis. For example, when modeling blood flow at the macroscopic level, researchers consider, in most cases, the blood as a Newtonian fluid. However, under certain experimental or physiological conditions, particularly at low shear rates, blood exhibits relevant non-Newtonian characteristics, and more complex constitutive models need to be used, see for instance \cite{MR3793259}. 

\medskip

  The first mathematical investigations of non-Newtonian fluid dynamics system go back to Ladyzhenskaya \cite{Lady-book}. Taking the following choice of constitutive law
\begin{equation}\label{S-Lad}
  \Ss=\mu |\D(\uu)|^{q-2}\D(\uu),   \quad \mu>0,
\end{equation}
the author developed the existence theory for weak solutions to non-stationary incompressible models with the aid of monotone operator theory, however, her result requires the following restriction $q\geq \frac{3d+2}{d+2}$ ($d$ is the dimension of space under consideration).  Since then, there has been significant progress in the mathematical theory of non-Newtonian fluids. To obtain the compactness of the convective term, some useful approximation techniques were used such as $L^{\infty}$  and Lipschitz truncations, and thus the existence of weak solutions with a wide range of the (constant) power-law index $q>\frac{2d}{d+2}$ was established, see \cite{Fr-Ma-St, Di-Ma-steady} in the case of steady flow, and \cite{Ma-Ne-uns, Di-Wo-unsteady} for unsteady flow. We refer the interested reader to \cite{malek-book} for an overview and a discussion of the relevant results. On the other hand, studying the regularity of weak solutions of such a system is quite challenging, and it has been the main subject of many research papers. Of course, this is due to the non-linearity effect that appeared in the constitutive law \eqref{S-Lad}. This issue becomes more challenging when the fluid has the ability to sustain local stresses up to certain yield stress $\tau^*$ and behaves like a generalized Newtonian fluid otherwise (case of Bingham fluid, see \eqref{S-bing} below). As this subject is of crucial importance for our study, and in order to keep our introduction as simple as possible, we leave the discussion on this point to Section \ref{Sec-NES}.
\medskip

To the best of our knowledge, there are very few mathematical studies on compressible non-Newtonian fluids dynamics equations. Looking for global weak solutions {\it \`a la Leray} to compressible non-Newtonian fluids dynamics equations with appropriate boundary and initial conditions is far from being well studied. The main difficulty is to get rid of nonlinear terms when passing to the limit in the approximate solution, which is caused by the lack of compactness  of the density and the velocity. For this reason, some weaker notions of solutions have been proposed for compressible non-Newtonian systems. In \cite{Ma-Ne-uns}, and using considerations from J. Ne\v{c}as, A. Novotn\'y and M. \v{S}ilhav\'y \cite{Nec-nov}, the authors studied measure-valued solutions for the same class of constitutive law  \eqref{S-Lad} treated by Ladyzhenskaya. A. Abbatiello and E. Feireisl in \cite{Abb-Feir} developed the notion of dissipative solutions for a more general class of generalized Newtonian fluids inspired by what P.-L. Lions has proposed for incompressible Euler system in dimension greater than or equal to three. Our objective at this stage is twofold. First, we prove a local well-posedness result to the so-called Power Law model. Precisely, we show the local existence and uniqueness of strong solutions to the following system
\begin{equation}\label{Power-Law}
\begin{array}{ccc}
\del_{t}\rho+\divv(\rho \uu)=0\vspace*{0.2cm}\\
\del_{t}(\rho \uu)+\divv(\rho \uu\otimes u)-\divv \Ss_\delta +\nabla p=\rho f
\end{array}
\end{equation}
where $\Ss_\delta$ being the stress tensor given by
\begin{equation}\label{Sdelta-PL}
\Ss_\delta=
2\mu \D(\uu)+\lambda \divv \uu\,\mathbb{I}+\tau^*(|\D(\uu)|^2 +\delta^2)^{\frac{q-2}{2}} \D(\uu),
\end{equation}
where $\mu$ and $\lambda$ are the so-called Lam\'e viscosity coefficients assumed here constants, $q$ is a number belongs to $[1,\infty[$, $\tau^*$ is the yield stress also assumed to be a constant, and $|\D(\uu)|^2$ is the Hilbert-Schmidt norm defined by (also called Frobenius norm)
$$|\D(\uu)|^2=\underset{i,j=1}{\overset{d}{\sum}}\D_{ij}^2\quad \quad \D(\uu)=(\D_{ij})_{1\leq i,j\leq d}.$$
The Power Law model provides an alternative to the Bingham plastic model for concentrated non-settling slurries. The Power Law model also describes the flow behavior of many polymer solutions although the relationship between the friction factor and the Reynolds number differs from that required to describe the behavior of concentrated slurries. On the other hand, this model constitutes a good approximation of the compressible Bingham system given below, whence its importance follows.
\medskip

Our method of proof builds up the framework developed by H. J. Choe and H. Kim for compressible Navier-Stokes in the case of Newtonian fluids \cite{Kim-SNS}. Compared to \cite{Kim-SNS}, one of the main difficulties in our study is to establish the elliptic regularity ($L^p-$estimates, $p\in (1,+\infty)$) for the following elliptic system with $(\delta,q)-$structure
\[
-\divv(\Ss_\delta)=f,
\]
over a periodic domain $\Omega$, and a compatibility condition on $f$. For Newtonian fluids, i.e. when $\tau^*=0$ in \eqref{Sdelta-PL}, the elliptic regularity of the above system follows from the well-known results  \cite{ag-do-ni2, Sol1}. Whereas for non-Newtonian fluids, the establishing of the $L^p$ regularity is highly non-trivial, and it relies on several parameters characterizing the nature of the fluid under consideration, see for instance Section \ref{Sec-NES}. The main task of this paper was to show the $L^p-$estimates for $\mu, \lambda, \tau^*, \delta$ and $q$ such that
\[
\mu>0 \quad 2\mu +\lambda>0 \quad \tau^*\geq 0 \quad \delta\geq 0 \quad q\geq 1.
\]
Unfortunately, the condition $\delta\geq c>0$ seems, at least using our method of proof, to be a necessary condition to establish the $L^p-$estimates ($p\neq 2$) when $q=1$ in dimension greater than or equal to two\footnote{For $p=2$, i.e. the $H^{2}-$estimate holds even if $\delta=0.$ See Proposition \ref{prop-reg}.} (more details are furnished in the next section). This restriction prevents us to study the limit of the sequence $(\rho_\delta, \uu_\delta)$ solution to \eqref{Power-Law}-\eqref{Sdelta-PL} for $q=1$ when $\delta$ tends to zero, and hence obtain a well-posedness result of the limit system in the realistic three-dimensional space. In one dimension space, this condition could be ignored as we will see and we obtain, as a second objective of this paper, a local well-posedness result of the so-called compressible Bingham system
\begin{equation}\label{Bing-sys}
\begin{array}{ccc}
\del_{t}\rho+\del_x(\rho \uu)=0\vspace*{0.2cm}\\
\del_{t}(\rho \uu)+\del_x(\rho \uu^2)-\del_x \Ss_0 +\del_x p=\rho f
\end{array}
\end{equation}
where $p=a\rho^\gamma$ and $\Ss_0$ being the stress tensor given as follows
\begin{equation}\label{S-bing}
\left\{
\begin{array}{lll}
\Ss_0=\mu \del_x \uu+\tau^*\dfrac{\del_x(\uu)}{|\del_x(\uu)|}\,\,&{\rm if}\,\, |\del_x\uu|\neq 0\\
|\Ss_0|<\tau^*\,\, &{\rm if}\,\, |\del_x\uu|= 0.
\end{array}
\right.
\end{equation}
 The Bingham fluid  behaves like a rigid body with a small stress intensity, and like a viscous fluid if the stress intensity exceeds the yield stress value $\tau^*$. The condition $\del_x \uu=0$ corresponds to the rigid part of a flow while the latter condition conforms to viscid flow. The modeling of Bingham materials is of crucial importance in industrial applications since a large variety of materials (e.g. foams, pastes, slurries, oils,
ceramics, etc.) exhibit the fundamental character of viscoplasticity, that is the capability of flowing only if the stress is above some critical value. Besides, we stress out that our result on compressible Bingham system extends the work of I. V. Basov and V. V. Shelukhin where vacuum is excluded initially; however, locally in-time, see  \cite{Ba-Sh-strong1d}.  
 \medskip

   The structure of the paper is the following.  In the next section, we present our main results, namely Theorem \ref{theo1} which shows the local well-posedness of Power Law system \eqref{Power-Law}-\eqref{Sdelta-PL}, and Theorem \ref{theo-2} which concerns the local well-posedness of Bingham system \eqref{Bing-sys}-\eqref{S-bing}. 
In Section \ref{Sec-NES}, we give an overview about elliptic regularity of nonlinear elliptic system with $(\delta,q)-$structure, and we present our contribution. Section \ref{Sec-The1} (resp. Sect. \ref{Sec-theo2}) is devoted to proving Theorem \ref{theo1} (resp. Theo. \ref{theo-2}). Finally, several tedious computations, which are skipped in the main body of this paper, will be provided in  Appendix  \ref{Appen-sec} for the sake of self-containedness.
\medskip

   After writing this paper, the author was informed by the recent work \cite{sarka}. Actually, in that paper the authors showed the existence and uniqueness of local in-time strong solutions for a more general class of non-Newtonian fluids than the Power Law model considered in the first part of this paper, however, using a different method of proof. More precisely, they worked with Lagrangian coordinates (here we use Eulerian coordinates), and they showed the short-time existence using the Weis multiplier theorem (here we use a fixed-point argument combined with the characteristics method and the Faeodo-Galerkin method). We claim that the result obtained in that paper could be also proved using our method of proof without too much effort. Mainly, we can show the local existence of a strong solution to a class of non-Newtonian fluids whose constitutive law is given as follows
\begin{equation}\label{gen-cons-law}
\Ss= 2\mu (|\D_d(\uu)|^2)\D_d(\uu) +\lambda(\divv \uu) \divv \uu \, \mathbb{I},
\end{equation}
with the following assumptions\footnote{These assumptions were imposed in \cite{sarka}.}
\[
\mu \in C^{3}([0,+\infty)) \qquad \lambda\in C^{2}(\RR),
\]
\[
\mu(s)+2s\mu^\prime(s)>0 \quad \mbox{for all}\quad s\geq 0 \qquad \lambda(r)+r\lambda^\prime(r)>0\quad \mbox{for all}\quad r\in \RR.
\]
The main ingredient will be the $L^p-$estimates on the associated elliptic system shown in \cite{sarka} under the above conditions. In this paper, we wanted to take a particular choice of constitutive law \eqref{gen-cons-law}, formulated in \eqref{Sdelta-PL}, first because the Power Law model is of its own independent interest, and second because it represents a good approximation of Bingham model \eqref{Bing-sys}-\eqref{S-bing}. Indeed, it is important to track the dependence of the a priori estimates established on $(\rho_\delta,\uu_\delta)$ solution of \eqref{Power-Law}-\eqref{Sdelta-PL} in terms of the parameter $\delta$ to be able to study the limit $\delta$ tends to zero. Certainly, the result shown in \cite{sarka} does not cover Bingham's fluid case.
\section{Main results}\label{Sec-MR}
In this section, we give first the definitions of strong solutions to system \eqref{Power-Law}-\eqref{Sdelta-PL} and to system \eqref{Bing-sys}-\eqref{S-bing}, then we state the main results of this paper, namely Theorem \ref{theo1} and Theorem \ref{theo-2}. 
\subsection{Well-posedness of Power Law system}
We state the following definition of a strong solution to system \eqref{Power-Law}-\eqref{Sdelta-PL}.
\begin{definition}\label{def-sol}
The pair $(\rho_\delta,\uu_\delta)$ is called a strong solution to system \eqref{Power-Law}-\eqref{Sdelta-PL} if $(\rho_\delta,\uu_\delta)$ is a weak solution, satisfying equations \eqref{Power-Law} almost everywhere in $(0,T^*)\times \Omega$, and enjoys the following properties 
\begin{align*}
&\rho_\delta\in L^{\infty}(0,T^*; W^{1,6}(\Omega))  \hspace*{1.7cm}(\rho_\delta)_t \in L^\infty(0,T^*; L^6(\Omega))\\
&\uu_\delta\in L^{\infty}(0,T^*; H^{1}(\Omega)) \hspace*{2cm} \nabla^2 \uu_\delta\in L^{2}(0, T^*; L^6(\Omega))\\
&\sqrt{\rho_\delta}(\uu_\delta)_t\in L^{\infty}(0, T^*; L^2(\Omega)) \hspace*{1cm}  (\uu_\delta)_t\in L^{2}(0, T^*; H^{1}(\Omega)).
\end{align*}
\end{definition}

The first main result of this paper is presented in the following theorem.
\begin{theorem}\label{theo1}
Let $\Omega=\mathbb{T}^3$ a periodic domain. Suppose that
$$\mu >0\quad \quad  2\mu +\lambda>0\quad q
\geq 1\quad \delta\geq c>0.$$
We assume moreover that the initial data $(\rho_0^\delta, \uu_0^\delta)$ and the external force $f$ satisfy the following 
\begin{equation}\label{cond-ini-PL}
0\leq \rho_0^\delta\in W^{1,6}(\Omega),\qquad \uu_0^\delta\in W^{2,6}(\Omega),\qquad f\in H^{1}(0,T; W^{1,6}(\Omega)),
\end{equation}
and the following compatibility condition holds 
\begin{equation}\label{comp-13}
-\divv\Big(2\mu \D(\uu_0^\delta)+\lambda\divv \uu_0^\delta\,\mathbb{I}+\tau^* \big(|\D(\uu_0^\delta)|^2+\delta^2 \big)^{\frac{q-2}{2}} \D(\uu_0^\delta)\Big)+\nabla p_0^\delta=\sqrt{\rho_{0}^\delta} g\quad \mbox{for a.e.}\quad x\in \Omega,
\end{equation}
where $g$ is a function in $ L^{6}(\Omega)$. Then, there exist a small time $T^*\in (0,T)$ and a unique strong  solution $(\rho_\delta, \uu_\delta)$ in the sense of Definition \ref{def-sol} to the initial value problem \eqref{Power-Law}-\eqref{Sdelta-PL}  with initial conditions \eqref{cond-ini-PL}.
\end{theorem}
The proof of this theorem is given in Section \ref{Sec-The1}.
\begin{remark}\label{chin-resul}
When $q\in (1,2)$, the local-in-time existence and uniqueness of strong solutions to system \eqref{Power-Law}-\eqref{Sdelta-PL} was shown in \cite{Hong-Xiao_diff} in the case of "pure nonlinear singular viscosity term", i.e. for $\mu=\lambda=\delta=0$; however, this result is limited to the one-dimensional space. We were not interested here to see if such result could be generalized to the multi-dimensional case since the assumption $q>1$ is critical in that case. Actually, the case when $q=1$ requires a specific attention, and it constitutes our interest in the second main theorem of this paper.
\end{remark}

\subsection{Well-posedness of Bingham system}
We state the following definition of a strong solution to system \eqref{Bing-sys}-\eqref{S-bing}.
\begin{definition}\label{def-sol-bing}
The pair $(\rho,\uu)$ is called a strong solution to system \eqref{Bing-sys}-\eqref{S-bing} if $(\rho,\uu)$ is a weak solution, satisfying \eqref{Bing-sys} almost everywhere in $(0,T^*)\times \Omega$, and enjoys the following properties 
\begin{align*}
&\rho\in L^{\infty}(0,T^*; W^{1,6}(\Omega))  \hspace*{1.2cm}\rho_t \in L^\infty(0,T^*; L^6(\Omega))\\
&\uu\in L^{\infty}(0,T^*; H^{1}(\Omega)) \hspace*{1.5cm} \del_x^2 \uu\in L^{2}(0, T^*; L^6(\Omega))\\
&\sqrt{\rho}\uu_t\in L^{\infty}(0, T^*; L^2(\Omega)) \hspace*{1cm}  \uu_t\in L^{2}(0, T^*; H^{1}(\Omega)).
\end{align*}
\end{definition}

The second main result of this paper is presented in the following theorem.
\begin{theorem}\label{theo-2}
Let $\Omega$ a periodic domain in $\mathbb{R}$. Suppose that $\mu >0$ and 
 the initial data $(\rho_0, \uu_0)$ and the external force $f$ satisfy the following regularity
\begin{equation}\label{cond-ini-bing}
0\leq \rho_0\in H^1(\Omega),\quad \uu_0\in W^{2,6}(\Omega),\quad f\in H^{1}(0,T; W^{1,6}(\Omega)),
\end{equation}
and the following compatibility condition holds 
\begin{equation}\label{comp-bing}
-\del_x\big(\Ss_0\big)+\del_x p_0=\sqrt{\rho_0}g \quad \mbox{for a.e.}\quad x\in \Omega,
\end{equation}
where $g$ is a function in $L^6(\Omega)$, and $\Ss_0$ is defined as follows
\begin{equation}\label{S0-comp}
\left\{
\begin{array}{lll}
\Ss_0=\mu \del_x \uu_0+\tau^*\dfrac{\del_x(\uu_0)}{|\del_x(\uu_0)|}\,\,&{\rm if}\,\, |\del_x\uu_0|\neq 0\\
|\Ss_0|<\tau^*\,\, &{\rm if}\,\, |\del_x\uu_0|= 0.
\end{array}
\right.
\end{equation}
Then there exist a small time $T^*\in (0,T)$ and a unique strong  solution $(\rho, \uu)$ in the sense of Definition \ref{def-sol-bing} to the initial value problem \eqref{Bing-sys}-\eqref{S-bing} with initial conditions given in \eqref{cond-ini-bing}.
\end{theorem}
The proof of this theorem is given in Section \ref{Sec-theo2}.
\begin{remark}
Our assumptions on the initial data and external force in Theorem \ref{theo-2} are not optimal. We did not try to optimize such assumptions for two reasons: first, the proof of Theorem \ref{theo-2} is based on the result presented in Theorem \ref{theo1} where all the estimates needed for such a result have been established under those assumptions. Second, as the one-dimensional case is not an interesting case in the physical point of view, then we prefer leave this issue to our forthcoming paper \cite{altaki-bing2} where we aim to extend the result presented in Theorem \ref{theo-2} to the multidimensional case (up to now to the two-dimensional case, see Remark \ref{rem-FS}). 
\end{remark}

\section{Nonlinear elliptic system with $(\delta,q)-$structure}\label{Sec-NES}
As evoked in the introduction, studying the elliptic regularity for a nonlinear elliptic system with $(\delta,q)-$structure depending on the symmetric part of the gradient constitutes one of the main ingredients of our proof of local well-posedness result to system \eqref{Power-Law}-\eqref{Sdelta-PL} and system \eqref{Bing-sys}-\eqref{S-bing}. Therefore, this section aims to give a short review of the state-of-the-art of such a problem and to present our contribution, namely, Proposition \ref{prop-reg} below. Indeed, consider 
the following nonlinear elliptic system over a periodic domain
\begin{align}
    \begin{split}\label{elliptic-ope-GENERAL}
        -\divv \Ss_\delta&=f,\\
         \Ss_\delta &=
2\mu \D(\uu)+\lambda \divv \uu\,\mathbb{I}+\tau^*(|\D(\uu)|^2 +\delta^2)^{\frac{q-2}{2}} \D(\uu),
     \end{split}
\end{align}
and assume that $f$ is a given function belongs to $L^p,\, 1<p<\infty,$  satisfying the compatibility condition $\int_\Omega f\,dx=0.$

\medskip
By dropping the divergence operator in the above system, i.e., taking $\divv \uu=0$ (case of incompressible fluids), we can find an extensive literature about the study of the regularity of weak solutions to such a system with different types of boundary conditions.  We do not want to comment on all of the existence results, and to avoid any further discussion about the complexity produced by the boundary conditions, we do not specify them instead we leave those various choices open as they may depend on the original problem. Indeed, we want to insist here on the conditions imposed on the parameters appeared in the stress $\Ss_\delta$ under which the $H^{2}-$estimate, or more generally, the $W^{2,p}-$estimates have been shown. The interested reader is referred to \cite{Fuchs-Seregin-Book, wielage2005analysis, aman-reg, Bo-Pr-Lp, Veiga, Ber-Die-Ruz, Bers-Ruz, Bers3} and the references given therein. Readers interested in a similar problem, precisely by p-Laplacian equation, are referred to the recent work by A. Cianchi and V. G. Maz'ya in \cite{Cia-mazya}. To summarize, the results established in these papers tell us
\begin{equation*}
{\rm if}\; f\in L^2(\Omega) \quad \mbox{and}\quad    \mu> 0 \quad \tau^*\geq 0, \quad \delta\geq 0 \quad q> 1, \qquad \mbox{then} \quad \uu\in W^{2,2}(\Omega).
\end{equation*}
Some of the papers mentioned above treated the case when $\mu=0.$  However, a further assumption on $f$ was imposed, and different regularity on $\uu$ has been shown depending on the domain under consideration. 
On the other hand, in \cite{Bo-Pr-Lp}, the authors studied the parabolic problem associated to \eqref{elliptic-ope-GENERAL}.
 Their study is based on maximal regularity theory for a suitable linear problem and a contraction argument. Assuming that the initial data $\uu\vert_{t=0}=\uu_0(x)$ belongs to $W^{2-\frac{2}{p},p}(\Omega)$ with $p>d+2$, then their result is summarized as follows
\begin{equation}\label{br-bot-res}
{\rm if}\; f\in L^p(\Omega) \quad \mbox{and}\quad   \mu\geq 0 \quad \tau^*>0, \quad \delta\geq c> 0  \quad q\geq  1, \qquad \mbox{then} \quad \uu\in W^{2,p}(\Omega).
\end{equation}
Clearly, we can remark that the case when $\delta=0$ and $q=1$ is not included in the above results. We shall see that in one dimension space, the $W^{2,p}-$estimates are still true even when $\delta=0$ and $q=1$, however, for $\mu>0,$ as well  the $H^{2}-$estimate in dimension 2 or 3. This is true even for compressible fluids, i.e., for $\divv u\neq 0$. Nevertheless, for dimension greater than one, the $W^{2,p}-$estimates require, at least using our method, the assumption $\delta \geq c>0.$

\subsection{Existence of weak solutions}
 Given $f\in L^{2}(\Omega).$ For $\delta>0,$ we say that $\uu\in H^{1}(\Omega)$ is a weak solution to \eqref{elliptic-ope-GENERAL} if for all $\varphi \in H^{1}(\Omega),$ we have
\begin{align*}
\intO \Ss_\delta : \nabla \varphi\,dx= \intO f\cdot \varphi\,dx.
\end{align*} 
The proof of existence and uniqueness of weak solutions to \eqref{elliptic-ope-GENERAL} is not so difficult, and it could be shown using a sort of compacity/monotonity arguments (cf. \cite{J-Lions-Q}), or using tools from  monotone operator theory (cf. \cite{Brezis-op}).

The constitutive law $\Ss_\delta$ constitutes a good approximation of Bingham's constitutive law $\Ss_0$ given in \eqref{S-bing} in the one-dimensional setting. The notion of weak solutions of \eqref{elliptic-ope-GENERAL} when $\Ss_\delta$ is replaced by $\Ss_0$ was firstly developed by G. Duvaut and J.-L. Lions in their book \cite{Du-Lions-B} by introducing the concept of variational inequality. Another variational approach was  later proposed, see for instance \cite{Sh-B-L, Fuchs-Seregin-Book}.

\subsection{Regularity of weak solutions} 
Let us assume that there exists a unique solution to system \eqref{elliptic-ope-GENERAL}. Then we have the following result.
\begin{proposition}\label{prop-reg} 
Let $\Omega$ a bounded domain with periodic boundary condition in $\mathbb{R}^d$. Then, the following assertions hold.
\medskip

\begin{itemize}
\item If $d=1$ and $f\in L^{p}(\Omega)$, then the solution $\uu$ of \eqref{elliptic-ope-GENERAL} is in $W^{2,p}(\Omega)$ provided $\mu>0,\, q\geq 1$, uniformly with respect to $\delta,$ that is it for $\delta\geq 0$. Particularly, we have
\[
\dfrac{\mu}{p} \intO |\del_x^2\uu|^p\,dx+\tau^*(q-1)\intO 
|\del_x\uu|^2(|\del_x \uu|^2+\delta^2)^{\frac{q-4}{2}}|\del_x^2 \uu|^{p}\,dx\leq \dfrac{\mu^{1-p}}{p} \intO |f|^p\,dx.
\] 

\item If $d=2$ or $3$, and $f\in L^{2}(\Omega)$,  then the solution $\uu$ of \eqref{elliptic-ope-GENERAL} is in $H^{2}(\Omega)$ provided $\mu>\varepsilon, 2\mu+\lambda>\varepsilon,$ for some small $\varepsilon>0, q\geq 1$, uniformly with respect to $\delta$, that is it for $\delta\geq 0$. Particularly, we have
\begin{align*}
\mu \intO &|\nabla \curl \uu |^2\,dx +(2\mu+\lambda) \intO |\nabla \divv \uu|^2\,dx \\
&+\tau^*\min(1,(q-1))\intO  |\D(\uu)|^2 (|\D(\uu)|^2+\delta^2)^{\frac{q-4}{2}} |\nabla \D(\uu)|^2\,dx\leq C \intO |f|^2\,dx.
\end{align*}

\item If $d=2$ or $3$, and $f\in L^{p}(\Omega)$,  then the linearized operator associated to equation \eqref{elliptic-ope-GENERAL} at a reference solution $\uu^*\in W^{2-\frac{2}{p},p}(\Omega),\; p>d+2,$ denoted by $\mathcal{A}(\uu^*, \D)$  still yield maximal $L^p-$regularity  provided  $\mu>0,\; 2\mu +\lambda>0, \; q\geq 1$; however, not uniformly with respect to $\delta$, i.e., for $\delta \geq c>0.$ Moreover, by perturbation argument, we get
\[
\| \uu \|_{W^{2,p}(\Omega)}\leq \|\mathcal{A}(\uu^*, \D) \uu\|_{L^p (\Omega)}\leq C(\delta^{-1}) \| f\|_{L^p (\Omega)}.
\]
\end{itemize}
\end{proposition}
\begin{proof}
The proofs of the first and second statements are based on the classical energy estimates method.
 We start by proving the first statement. Indeed, in one-dimensional space, equation \eqref{elliptic-ope-GENERAL} becomes
$$-\del_x\Big(\mu \del_x \uu  +\tau^* (|\del_x \uu|^2 +\delta^2)^{\frac{q-2}{2}} \del_x \uu\Big)=f.$$
Multiplying the left-hand side term of the above equation with $-|\del_x^2 \uu|^{p-2}\del_x^2 \uu$ and integrate over $\Omega$, we obtain
 \begin{align*}
 \mu &\intO |\del_x^2 \uu|^{p}\,dx+\tau^*\intO \del_x \big((|\del_x \uu|^2 +\delta^2)^{\frac{q-2}{2}} \del_x \uu \big)|\del_x^2 \uu|^{p-2}\del_x^2 \uu\,dx\\
 &= \mu \intO |\del_x^2 \uu|^{p}\,dx+\tau^* \intO (|\del_x \uu|^2 +\delta^2)^{\frac{q-2}{2}} |\del_x^2 \uu|^p \,dx +(q-2)\tau^*\intO (|\del_x \uu|^2 +\delta^2)^{\frac{q-4}{2}} |\del_x \uu|^2|\del_x^2 \uu|^p \,dx\\
 &=  \mu \intO |\del_x^2 \uu|^{p}\,dx+\tau^*  \intO (|\del_x \uu|^2 +\delta^2)^{\frac{q-4}{2}}\big[(q-1)|\del_x\uu|^2+\delta^2 \big]|\del_x^2\uu|^p\,dx\\
 &\geq \mu \intO |\del_x^2 \uu|^{p}\,dx+\tau^*(q-1)\intO 
|\del_x\uu|^2(|\del_x \uu|^2+\delta^2)^{\frac{q-4}{2}}|\del_x^2 \uu|^{p}\,dx
 \end{align*}
 Thus, we deduce that
\begin{align*}
\mu \intO |\del_x^2 \uu|^{p}\,dx+\tau^*(q-1)\intO 
|\del_x\uu|^2(|\del_x \uu|^2+\delta^2)^{\frac{q-4}{2}}|\del_x^2 \uu|^{p}\,dx\leq -\intO f \,|\del_x^2 \uu|^{p-2}\del_x^2 \uu\,dx. 
\end{align*}
Now, applying successively H\"older and Young inequalities on the right-hand side term
\begin{align*}
 \intO f \,|\del_x^2 \uu|^{p-2}\del_x^2 \uu\,dx
&\leq \Big(\intO |f|^p\,dx\Big)^{1/p}\, \Big(\intO|\del_x^2\uu|^p\,dx\Big)^{(p-1)/p}\\
&\leq \dfrac{\mu^{1-p}}{p} \intO |f|^p\,dx +\dfrac{\mu(p-1)}{p} \intO |\del_x^2\uu|^p\,dx.
\end{align*}
Finally, we have proved that
\begin{align*}
\dfrac{\mu}{p} \intO |\del_x^2\uu|^p\,dx+\tau^*(q-1)\intO 
|\del_x\uu|^2(|\del_x \uu|^2+\delta^2)^{\frac{q-4}{2}}|\del_x^2 \uu|^{p}\,dx\leq \dfrac{\mu^{1-p}}{p} \intO |f|^p\,dx.
\end{align*}
This finishes the proof of the first statement.

Let us move to prove the second statement. Indeed, firstly we rewrite equation \eqref{elliptic-ope-GENERAL} as follows
\[
-\mu \Delta \uu - (\lambda+\mu)\nabla \divv \uu -\tau^*\divv\big((|\D(\uu)|^2 +\delta^2)^{\frac{q-2}{2}} \D(\uu)\big)=f.
\]
Taking the scalar product of the above equation by $-\Delta \uu,$ and integrate over $\Omega,$ we deduce
\begin{align}\label{H2E1}
\begin{split}
\mu \intO | \Delta \uu |^2\,dx& + (\lambda+\mu) \intO  \nabla \divv \uu \cdot \Delta \uu \,dx \\
&+\tau^*\intO  \divv\big((|\D(\uu)|^2 +\delta^2)^{\frac{q-2}{2}} \D(\uu)\big)\cdot \Delta \uu \,dx=-\intO  f\cdot \Delta \uu \,dx.
\end{split}
\end{align}
Twice integration by parts over a periodic domain leads to
\begin{align*}
& \intO  \nabla \divv \uu \cdot \Delta \uu \,dx= \intO |\nabla \divv \uu |^2\,dx\\
& \intO \divv\big((|\D(\uu)|^2 +\delta^2)^{\frac{q-2}{2}} \D(\uu)\big)\cdot \Delta \uu \,dx= \intO \nabla\big((|\D(\uu)|^2 +\delta^2)^{\frac{q-2}{2}} \D(\uu)\big): \nabla \D(\uu) \,dx
\end{align*}
Besides, denoted by 
$$F_\delta(A)=(|A|^2 +\delta^2)^{\frac{q-2}{2}} A \qquad \quad A\in \mathbb{R}^{d\times d}\quad d=2,3,$$
then we compute
\begin{align*}
\dfrac{\del F_\delta(A)_{kl}}{\del A_{ij}}=(|A|^2 +\delta^2)^{\frac{q-2}{2}} \delta^{ij}_{kl}+(q-2)(|A|^2 +\delta^2)^{\frac{q-4}{2}} A_{kl} A_{ij}.
\end{align*}
Consequently, we have
\begin{align*}
\dfrac{\del F_\delta(A)_{kl}}{\del x_m}\dfrac{\del A_{kl}}{\del x_m}&=\dfrac{\del F_\delta(A)_{kl}}{\del A_{ij}}\dfrac{\del A_{ij}}{\del x_m}\dfrac{\del A_{kl}}{\del x_m} \\
&= (|A|^2 +\delta^2)^{\frac{q-2}{2}} \delta^{ij}_{kl}\dfrac{\del A_{ij}}{\del x_m}\dfrac{\del A_{kl}}{\del x_m} +(q-2)(|A|^2 +\delta^2)^{\frac{q-4}{2}}  A_{kl} A_{ij}\dfrac{\del A_{ij}}{\del x_m}\dfrac{\del A_{kl}}{\del x_m}\\
&=(|A|^2 +\delta^2)^{\frac{q-4}{2}}\Big((|A|^2+\delta^2)\dfrac{\del A_{ij}}{\del x_m}\dfrac{\del A_{ij}}{\del x_m}+(q-2)A_{ij} A_{kl}\dfrac{\del A_{ij}}{\del x_m}\dfrac{\del A_{kl}}{\del x_m}\Big).
\end{align*}
Thus, we conclude that
\[
\intO  \divv\big((|\D(\uu)|^2 +\delta^2)^{\frac{q-2}{2}} \D(\uu)\big)\cdot \Delta \uu \,dx\geq \min(1,(q-1))\intO  |\D(\uu)|^2 (|\D(\uu)|^2+\delta^2)^{\frac{q-4}{2}} |\nabla \D(\uu)|^2\,dx.
\]
Finally, from equation \eqref{H2E1}, and the above estimate we deduce
\begin{align*}
\mu \intO &| \Delta \uu |^2\,dx + (\lambda+\mu) \intO  |\nabla \divv \uu |^2 \,dx \\
&+\tau^* \min(1,(q-1))\intO  |\D(\uu)|^2 (|\D(\uu)|^2+\delta^2)^{\frac{q-4}{2}} |\nabla \D(\uu)|^2\,dx\leq -\intO f \cdot \Delta \uu \,dx.
\end{align*}
The right-hand side term is estimated using H\"older and Cauchy-Schwarz inequalities as follows
\[
\intO f \cdot \Delta \uu \,dx \leq \varepsilon \intO |\Delta \uu|^2 \,dx + C(\varepsilon) \intO |f|^2\,dx.
\]
Therefore, we obtain
\begin{equation}\label{est-bef}
(\mu-\varepsilon) \intO | \Delta \uu |^2\,dx + (\lambda+\mu) \intO  |\nabla \divv \uu |^2 \,dx \leq \intO |f|^2\,dx.
\end{equation}
To look at the optimal condition such that the $H^2-$estimate on $\uu$ holds through this method, one has to benefit from the following identity
\[
\curl \curl \uu = \nabla \divv \uu -\Delta \uu,
\]
Consequently, we rewrite estimate \eqref{est-bef} as follows
  \begin{align*}
 (\mu-\varepsilon) \intO | \nabla \curl \uu |^2\,dx + (\lambda+2\mu-\varepsilon) \intO  |\nabla \divv \uu |^2 \,dx \leq \intO |f|^2\,dx.
  \end{align*}
The proof of the second statement is now finished.

We give two different proofs of the last statement. Both of them are based on showing the {\it strong ellipticity} condition\footnote{In some monograph, it is called Legendre-Hadamard condition.} of the linearized operator, say, $\mathcal{A}(\uu^*, \D)$ at a reference solution $\uu^*$, and a perturbation argument. This last step will be not given in details here, however, the details are postponed to the next section.  Indeed, we start by the short proof:

\hfill\break
{\it First proof.} Denoting by $\mathcal{A}(\uu, \D)$ the elliptic operator associated to \eqref{elliptic-ope-GENERAL} (see \eqref{defA} below). We showed in the last statement that the operator $\mathcal{A}(\uu^*, \D)$ acting from\footnote{This fact is true also for the nonlinear operator, however, in order to apply the result by Geymonat \cite{Geym}, the operator must be quasi-linear. }
\[
\mathcal{A}(\uu^*, \D): H^2(\Omega)\rightarrow L^2(\Omega),
\]
is onto and so $\dim (\coker(\mathcal{A}(\uu^*, \D)))=0.$ Moreover, by uniqueness of solutions, we have $\dim (\ker (\mathcal{A}(\uu^*, \D)))=0$ for all $p\geq 2$. Assuming that $\delta\geq c>0$, then it was proved by Geymonat in \cite{Geym} that the index of $\mathcal{A}(\uu^*, \D)$ defined as follows
\[
{\rm index}(\mathcal{A}(\uu^*, \D)):= \dim (\ker(\mathcal{A}(\uu^*, \D)))-\dim (\coker(\mathcal{A}(\uu^*, \D)))
\]
is independent of $p$, $1<p<\infty$. As ${\rm index}(\mathcal{A}(\uu^*, \D))=0$ for $p=2$, then it is equal to zero for all $p\geq 2$. Thus, $\dim (\coker(\mathcal{A}(\uu^*, \D)))=0$ for all $p\geq 2,$ i.e., $\mathcal{A}(\uu^*, \D)$ is onto for all $p\geq 2$, which proves that if $v$ solution of $-\mathcal{A}^*(\uu^*, \D)v=h$ with $h\in L^p(\Omega)$ and $\dashint_\Omega h\,dx=0$, then $v\in W^{2,p}(\Omega)$.
Rewriting \eqref{elliptic-ope-GENERAL} as follows
\[
-\mathcal{A}(\uu^*, \D) \uu=\mathcal{A}(\uu, \D)\uu-\mathcal{A}(\uu^*, \D)\uu +f,
\]
and applying a perturbation argument, we deduce the desired estimate. The details are skipped here. The constant $C$ can be chosen independently of $\uu^*$ due to the uniform continuity of the coefficients, see \cite[Theorem 5.7]{Denk-Book}. Notice that, even though the $H^2$-estimate for the operator $\mathcal{A}(\uu^*, \D)$ holds uniformly with respect to $\delta,$ we have to assume that $\delta \geq c>0$ to fulfill the assumptions imposed on the coefficients in \cite[Theorem 1.1]{Geym}.

\hfill\break
{\it Second proof.}  We shall prove that the linearized operator $\mathcal{A}(\uu^*,\D)$ satisfies  the strong ellipticity condition. Therefore, the regularity property of the nonlinear operator $\mathcal{A}(\uu,\D)$ follows after applying a perturbation argument as it will be shown in the next section. To start, let us  rewrite the elliptic operator as follows
\begin{align*}
\divv \Ss_\delta&=\divv\Big(2\mu \D(\uu)+\lambda\divv \uu\,\mathbb{I}+\tau^* (|\D(\uu)|^2 +\delta^2)^{\frac{q-2}{2}} \D(\uu)\Big)\\
&=\Big(\mu +\dfrac{\tau^*}{2}(|\D(\uu)|^2 +\delta^2)^{\frac{q-2}{2}}\Big)\Delta \uu +\Big(\lambda+\mu +\dfrac{\tau^*}{2}(|\D(\uu)|^2 +\delta^2)^{\frac{q-2}{2}}\Big)\nabla \divv \uu\\
&\quad+\tau^* \dfrac{(q-2)}{2}(|\D(\uu)|^2 +\delta^2)^{\frac{q-4}{2}}\nabla\big( |\D(\uu)|^2\big)\cdot \D(\uu).
\end{align*}
The $j$-component of the last vector on the right-hand side is given as follows
\[
\tau^* (q-2)\underset{k}{\sum}\underset{lm}{\sum} (|\D_{lm}(\uu)|^2 +\delta^2)^{\frac{q-4}{2}} (\D_{lm}(\uu)\del_k \D_{lm}(\uu))\D_{kj}(\uu),
\] 
where we denoted by
$$\del_j (|\D(\uu)|^2)=2\underset{k,l=1}{\overset{n}{\sum}}\D_{kl}(\uu)\del_j \D_{kl}(\uu).$$
Define the function $\beta: \RR^+\rightarrow \RR^+$ as
\begin{equation}\label{def-gamma}
\beta(s)=\mu +\dfrac{\tau^*}{2}s^{\frac{q-2}{2}},
\end{equation}
and denote by $B:= |\D(\uu)|^2 +\delta^2$, then we can write 
$$\divv \Ss_\delta=\beta(B)\Delta \uu +\big(\lambda+\beta(B)\big)\nabla \divv \uu +2\beta^{\prime}(B)\nabla\big( |\D(\uu)|^2\big)\cdot \D(\uu).$$
Then using the fact that $\D$ is symmetric, the $i^{th}$ entry of $\divv \Ss_\delta$ becomes 
\begin{align*}
\big[\divv \Ss_\delta\big]_i&= \underset{k=1}{\overset{3}{\sum}}\big(\beta(B)\del_k^2 \uu_i+(\lambda+\beta(B))\del_i\del_k \uu_k\big)+4\beta^{\prime}\underset{j,k,l=1}{\overset{3}{\sum}}\D_{ij} \D_{kl} \del_j \D_{kl}\\
&=\underset{k=1}{\overset{3}{\sum}}\big(\beta(B)\del_k^2 \uu_i+(\lambda+\beta(B))\del_i\del_k \uu_k\big)+4\beta^{\prime}(B)\underset{j,k,l=1}{\overset{3}{\sum}}\D_{ik} \D_{jl} \del_k\del_l \uu_{j}\\
&=\underset{j,k,l=1}{\overset{3}{\sum}}a_{ij}^{kl}\del_k\del_l\uu_j.
\end{align*}
We have set
\begin{equation}\label{defaij}
a_{i,j}^{k,l}=\beta(B) \delta_{kl}\delta_{ij}+(\lambda+\beta(B))\delta_{il}\delta_{jk}+4\beta^{\prime}(B)\D_{ik}\D_{jl},
\end{equation}
where $\delta_{kl}$ denotes the Kronecker symbol. Define the quasi-linear differential operator $\mathcal{A}(\uu,D)$ as
\begin{equation}\label{defA}
\mathcal{A}(\uu,\D)=\underset{k,l=1}{\sum}A^{k,l}\D_k \D_l,
\end{equation}
where the matrix-valued coefficients
\begin{equation*}
A^{k,l}(\uu)=\big(a_{i,j}^{k,l}\big).
\end{equation*}
Now, if $\uu\in W^{2,6}(\Omega)$, then by Sobolev embedding $\uu \in C^{1}(\overline{\Omega})$,
said the coefficients of the differential operator $\mathcal{A}(x,D)=\mathcal{A}(\uu(x),\D)$ are uniformly bounded continuous since $\delta\geq c>0$. Freezing $\mathcal{A}$ at a reference solution $\uu^*$, we obtain the linear operator
$$\mathcal{A}(\uu^*, \D)=\underset{k,l=1}{\sum}A^{k,l}_*\D_k \D_l.$$
Investigating the above problem as a quasi-linear equation.
For that, we notice that by the definition of $a_{ij}^{kl}$ in \eqref{defaij} and using the sum convention, we have for $\xi \in \mathbb{R}^n$, $\eta \in \mathbb{C}^n, \, |\xi|=|\eta|=1,$ 
\begin{align*}
(\mathcal{A}(x,\xi)\eta,\eta)&=\underset{i,j,k,1=1}{\overset{n}{\sum}}a_{ij}^{kl}\xi_k\eta_l \xi_i \bar{\eta_j}
\\
&=\beta(B) \underset{i,k=1}{\overset{n}{\sum}}\xi_k^2 \eta_i \bar{\eta_i}+(\lambda+\beta(B))\underset{i,k=1}{\overset{n}{\sum}}\xi_i \eta_k\xi_k  \bar{\eta_i}+4\beta^{\prime}(B)\underset{i,k=1}{\overset{n}{\sum}} d_{ik}\xi_k \eta_i \underset{j,l=1}{\overset{n}{\sum}}d_{jl}\xi_l \bar{\eta_j}
\\
&=\beta(B) |\xi|^2 |\eta|^2 +(\lambda+\beta(B))(\xi,\eta)\,(\eta, \xi)+4\beta^{\prime}(B)(\D\xi,\eta ) \overline{(\D\xi,\eta )}\\
&=\beta(B) |\xi|^2 |\eta|^2 +(\lambda+\beta(B))|(\xi,\eta)|^2+4\beta^{\prime}(B)|(\D\xi,\eta)|^2,
\end{align*}
where $(\cdot , \cdot)$ represents the inner product in $\mathbb{C}^{n\times n}$. In particular, we have $(\mathcal{A}(x,\xi)\eta, \eta)$ is real. Evidently, $\mathcal{A}(x,D)$ is strongly elliptic if 
$$\beta(B)>0,\; \lambda+\beta(B)\geq 0 \quad \mbox{and}\quad \beta^{\prime}(B)\geq 0.$$
In the case where $\beta^{\prime}(B)<0$, which corresponds to the case $1\leq q<2$, we infer from the Cauchy-Schwarz inequality 
\begin{align*}
{\rm Re}\, (\mathcal{A}(x, \xi)\eta, \eta)&\geq \beta(B)|\xi|^2 |\eta|^2 +(\lambda+\beta(B) )|(\xi,\eta)|^2+4 \beta^{\prime}(B)|\D(x)|^2|(\xi,\eta)|^2\\
&\geq \beta(B)|\xi|^2 |\eta|^2+\big(\lambda+\beta(B)+4\beta^\prime(B)|\D(x)|^2\big)|(\xi,\eta)|^2.
\end{align*}
Now if we have
$$\lambda+\beta(B) +4 \beta^{\prime}(B)|\D(x)|^2>0,$$
then the operator is strongly elliptic. If else, using again Cauchy-Schwarz inequality and the fact that $|\xi|=|\eta|=1$ we deduce that
\begin{equation*}
\beta(B)>0\quad 2\beta(B)+\lambda+4(B-\delta^2)\beta^{\prime}(B)>0 ,
\end{equation*}
is sufficient to have the strong ellipticity. The above condition is equivalent to
\begin{equation*}
\beta(B)=\mu +\dfrac{\tau^*}{2}B^{\frac{q-2}{2}}>0 \quad\quad  2\beta(B)+\lambda+4(B-\delta^2)\beta^{\prime}(B)=2\mu +\lambda+\tau^*B^{\frac{q-4}{2}}\big((q-1)|\D(\uu)|^2+\delta^2\big)> 0,
\end{equation*}
which is satisfied if we assume that
\begin{equation}\label{cond-strong-ell}
 \mu>0, \quad 2\mu +\lambda>0,\quad q\geq 1, \quad \mbox{and} \quad \delta \geq 0,   
\end{equation}
This ends the proof of Proposition \ref{prop-reg}.
\end{proof}

\begin{remark}
We emphasize that \eqref{cond-strong-ell} is sufficient to ensure the strong ellipticity condition of the linearized operator $\mathcal{A}(\uu^*, \D)$.
However, \eqref{cond-strong-ell} is not sufficient to ensure the uniformly bounded continuous condition of the coefficients of $\mathcal{A}(\uu^*, \D)$, which is a necessary assumption to apply the method introduced in \cite{Denk1-Max-Reg}. When using the energy estimates method, the last condition is, somehow, not involved, and for this reason, we were able to show the first two statements of Proposition \ref{prop-reg} without assuming $\delta$ to be far from zero. Of course, considering $\Omega=\mathbb{T}^3$ played an essential role in the energy estimates method.
\end{remark}
\begin{remark}\label{rem-FS}
Unfortunately, even with the uniform $H^{2}$-estimate of $\uu$ (with respect to $\delta$), we are not able to close the loop used in the proof of the existence of strong solutions to the Power Law model, and, consequently, studying the limit when $\delta$ tends to zero of the sequence $(\rho_\delta, \uu_\delta)$ for $q=1$ in \eqref{Sdelta-PL}. The main problem is to control the density in $L^\infty((0,T)\times \Omega)$, which requires $\int_0^T \| \divv \uu\|_{L^\infty}<+\infty$.  The $W^{2,p}-$estimates solve this issue. However, such estimates are not expected for the solution to the limit system in dimensions greater than or equal to three. In fact, apart from the two-dimensional case,  M. Fuchs and G. Seregin claim that the solution of the steady incompressible Bingham equations is not smooth, in the sense that $\uu$ is not of class $C^1$ near the set $\{x\in \Omega, |\D\uu(x)|=0\}.$ Precisely, there is an open subset $\Omega^*$ of $\Omega$ such that the solution of the steady imcompressible Bingham equations with external force $f\in L^{2, d-2+2\nu}(\Omega)$ belongs to $C^{1,\alpha}(\Omega^*; \RR^d)$ for any $\alpha<\nu$ and $\mathcal{H}^{d-2}(\Omega-\Omega^*)=0$. Here, $\mathcal{H}^{d-2}$ denotes the $(d-2)$ dimensional Hausdorff measure. In particular, if $d=2$ then we have no singular points, i.e. $\Omega^*=\Omega$. Our objective in a forthcoming paper \cite{altaki-bing2} is to extend the tools developed in \cite{Fuchs-Seregin-Book} to the unsteady compressible flow and prove the existence of solutions to the Bingham system in the two-dimensional setting.
\end{remark}
\section{Power law model: Proof of Theorem \ref{theo1}}\label{Sec-The1}
This section is devoted to prove Theorem \ref{theo1}. 
In this proof, we shall deal with several difficulties. First, the initial density may vanish in an open subset of $\Omega$, i.e. an initial vacuum may exist. So, we can perturb the initial density so that we avoid the vacuum to bring us difficulties. Second, as we are looking for strong solutions, then it is natural to establish higher estimates on the approximate solutions. We shall use our results established in Section \ref{Sec-NES}. 
 This section is then organized as follows. In the first subsection, we shall establish the necessary a priori estimates on the approximate solutions $(\rho_\delta, \uu_\delta)$. After that, we deal with the construction of such approximate solutions. In the last subsection, we prove the uniqueness of solutions.  
\subsection{A priori estimates on $(\rho_\delta, \uu_\delta)$.}
In this subsection, we will assume that the system \eqref{Power-Law}-\eqref{Sdelta-PL} has a smooth solution $(\rho_\delta,\uu_\delta)$
provided that the following compatibility condition on the initial data $(\rho_{0}^\delta,\uu_{0}^\delta)$ holds
\begin{equation}\label{comp-13'}
-\divv\Big(2\mu \D(\uu_{0}^\delta)+\lambda\divv \uu_{0}^\delta\,\mathbb{I}+\tau^* (|\D(\uu_{0}^\delta)|^2 +\delta^2)^{\frac{q-2}{2}} \D(\uu_{0}^\delta)\Big)+\nabla p_{0}^\delta=\sqrt{\rho_{0}^\delta}\,g\quad \mbox{for a.e.}\quad x\in \Omega,
\end{equation}
where $g$ is a function in $L^{6}(\Omega)$. We shall then establish some a priori estimates on it in higher norms. We define the following function
\begin{equation}\label{14-ch}
\psi(t)=1+\|\nabla \uu_\delta\|_{L^{2}(\Omega)}^{2}+||\sqrt{\rho_\delta}\uu_t||_{L^{2}(\Omega)}^{2}+||p_\delta(t)||_{W^{1,6}(\Omega)}.
\end{equation}
As was shown in \cite{Kim-SNS}, the establishing of the a priori estimates will be based on the boundedenes of the function $\psi(t)$. Observe that, by the usual Sobolev inequality
\[
\| \rho_\delta(t) \|_{L^{\infty}(\Omega)}+ \| p_\delta(t) \|_{L^{\infty}(\Omega)}\leq C  \psi(t).
\]  
Notice that in all the estimates established below, we will denote by $C$ or $c$ a generic constant depending on $T$ and the norms of the data, however, it does not depend on the parameter $\delta$. Just the notation $C(\delta^{-1})$ concerns a constant who may have an unfavorable effect when $\delta$ becomes close to zero. This notation appears when we are looking for the $L^p-$estimates on the elliptic equation \eqref{ellip-prob}. We emphasize that in one-dimensional space, all the estimates established in this section holds uniformly with respect to $\delta$ due to the first statement in Proposition~\ref{prop-reg}. This fact is crucial to study the limit of the sequence $(\rho_\delta,\uu_\delta)$ when $\delta$ tends to zero; see Section \ref{Sec-theo2}.

For the sake of simplicity, we will drop the index $\delta$ on all the equations and estimates written in this subsection.

In the sequel, we keep using the following notation
\begin{equation}\label{def-B}
B:=|\D(\uu)|^2+\delta^2.
\end{equation}
\hfill \break
{\bf Step 1}.\textit{ Basic energy estimate.}
We begin with elementary considerations. By virtue of the continuity equation \eqref{Power-Law}$_1$, we deduce the conservation of mass
\begin{equation}\label{cons-mas}
\int_{\Omega}\rho (t)\,dx=\int_{\Omega}\rho_{0}\,dx\;\quad(t\geq 0).
\end{equation}
Next, taking the scalar product of the momentum equation \eqref{Power-Law}$_2$ by $\uu$, integrating by parts over $\Omega,$ and then using the continuity equation \eqref{Power-Law}$_1$, we obtain
\begin{equation}\label{ener-est}
\dfrac{1}{2}\dfrac{d}{dt}\intO( \rho |\uu|^{2}+\dfrac{a}{\gamma-1}\rho^{\gamma})\,dx+\mu \intO |\nabla \uu|^{2}\,dx+(\lambda+\mu)\intO |\divv\uu|^2\,dx +\tau^*\int_{\Omega}B^{\frac{q-2}{2}}|\D(\uu)|^{2}\,dx= \intO \rho f\cdot \uu\,dx.
\end{equation}
Moreover, it follows by virtue of H\"{o}lder inequality that
\begin{equation*}
\intO \rho f\cdot \uu\,dx\leq ||f||_{L^{\frac{2\gamma}{\gamma-1}}(\Omega)}||\rho||_{L^{\gamma}(\Omega)}^{1/2}||\sqrt{\rho}\uu||_{L^{2}(\Omega)}\leq C ||f||_{L^{\frac{2\gamma}{\gamma-1}}(\Omega)}(1+||\rho||_{L^{\gamma}(\Omega)}^{\gamma}+||\sqrt{\rho}\uu||_{L^{2}(\Omega)}^{2}).
\end{equation*}
Now we apply Gronwall's inequality to deduce the following energy estimate
\begin{equation}\label{eq15-c}
\underset{0\leq t\leq T}{\sup} \intO(\rho |\uu|^{2}+p)\,dx+\intTO |\nabla \uu|^{2}\,dx\,dt\leq C.
\end{equation}
Independently, using Sobolev inequality we have
\begin{equation}\label{ch-17}
||\rho(t)||_{L^{q}(\Omega)}+||p(t)||_{L^{q}(\Omega)}\leq C \psi(t)\;\;\;{\rm for \;all}\;\;\;1\leq q\leq \infty.
\end{equation} 
\hfill \break
{\bf Step 2}. \textit{Estimate for $||\nabla \uu||_{L^{2}}$.} Taking the scalar product of the momentum equation \eqref{Power-Law}$_{2}$ by $\uu_t$, integrating over $\Omega$, we obtain
\begin{align}\label{188-ch}
\begin{split}
\int_{\Omega}\rho &|\uu_t|^{2}\,dx+\dfrac{1}{2}\dfrac{d}{dt}\int_{\Omega}\big(\mu |\nabla \uu|^{2}+(\lambda+\mu)(\divv \uu)^2+\dfrac{2}{q}\tau^* B^{\frac{q}{2}}\big)\,dx\\
&=\intO \big(\rho f-\rho\uu\cdot\nabla \uu\big)\cdot\uu_t\,dx-\intO \nabla p\cdot \uu_t\,dx.
\end{split}
\end{align}
Using H\"older and Young inequalities, we estimate the first term on the right-hand side of equation \eqref{188-ch} as follows
\begin{align*}
\intO \big(\rho f-\rho\uu\cdot\nabla \uu\big)\cdot\uu_t\,dx\leq  \int_{\Omega} (\rho|f|^{2}+\rho|\uu|^{2}|\nabla \uu|^{2})\,dx+\dfrac{1}{2}\intO \rho |\uu_t|^2\,dx.
\end{align*}
Thus after making an integration by parts on the last term of equation \eqref{188-ch}, we can write
\begin{align}\label{18-ch}
\begin{split}
\dfrac{1}{2}\int_{\Omega}&\rho|\uu_t|^{2}\,dx+\dfrac{1}{2}\dfrac{d}{dt}\int_{\Omega}\big(\mu |\nabla \uu|^{2}+(\lambda+\mu)(\divv \uu)^2+\dfrac{2}{q}\tau^* B^{\frac{q}{2}}\big)\,dx\\
&\leq \int_{\Omega} (\rho|f|^{2}+\rho|\uu|^{2}|\nabla \uu|^{2})\,dx+\int_{\Omega}p\divv \uu_{t}\,dx\\
&\leq \int_{\Omega} (\rho|f|^{2}+\rho|\uu|^{2}|\nabla \uu|^{2})\,dx+\dfrac{d}{dt}\int_{\Omega}p\divv \uu\,dx-\int_{\Omega}p_{t}\divv \uu\,dx.
\end{split}
\end{align}
Using the continuity equation, we have
\begin{equation}\label{2.8-d}
p_{t}+\divv(p\uu)+(\gamma-1)p\divv \uu=0,
\end{equation}
and hence
\begin{align*}
\int_{\Omega}p_{t}\divv \uu\,dx&=-\int_{\Omega}(\divv (p\uu))\divv \uu\,dx-(\gamma-1)\int_{\Omega}p(\divv \uu)^{2}\,dx\\
&=\int_{\Omega}(p\uu\cdot\nabla \divv \uu-(\gamma-1)p(\divv \uu)^{2})\,dx.
\end{align*}
Substituting this identity into equation \eqref{18-ch}, integrating over $(0,t)$, one obtains  
\begin{equation}\label{19ch}
\begin{split}
 \int_{0}^{t}\int_{\Omega}&\rho \uu_t^2\,dx\,ds+ \int_{\Omega} |\nabla \uu(t)|^2\,dx \\
& \leq C + C\intO p(t)^2\,dx +C\int_0^t \psi(s)\,ds\\
& \quad+C\int_{0}^{t}\int_{\Omega}\big(\rho|\uu|^2|\nabla \uu|^2 +p|\uu||\nabla \divv \uu|
+p(\divv \uu)^2\big)\,dx\,ds.
\end{split}
\end{equation}
Again in view of \eqref{eq15-c} and \eqref{ch-17}, we conclude that
\begin{equation*}
\intO p(t)^2\,dx = \intO p(0)^2 + \int_0^t\frac{\partial}{\partial s}\biggl{(}\intO p^2\,dx\biggr{)}\,ds
\le C + C\int_0^t\psi(s)^4\,ds.
\end{equation*}
In order to estimate the remaining terms in \eqref{19ch}, one has to use the elliptic estimates proved in Proposition~\ref{prop-reg}. Indeed, we know that $\uu$ satisfies the following equation
\begin{equation}\label{ellip-prob}
-\divv (\Ss_\delta) =-\rho \uu_t-\rho (\uu\cdot\nabla \uu)+\rho f-\nabla p.
\end{equation}
By virtue of Proposition \ref{prop-reg}, we get
\begin{align*}
\|\nabla^{2}\uu\|_{L^{2}(\Omega)}&\leq  C(\delta^{-1})(\|\rho \uu_t\|_{L^{2}(\Omega)}+\|\rho \uu\cdot\nabla \uu\|_{L^{2}(\Omega)}+\|\rho f\|_{L^{2}(\Omega)}+\|\nabla p\|_{L^{2}(\Omega)})\\
&\leq C(\|\rho \|_{L^{\infty}(\Omega)}^{1/2}\|\sqrt{\rho}\uu_t\|_{L^{2}(\Omega)}+\|\rho\|_{L^{\infty}(\Omega)}\|\uu\|_{L^{6}(\Omega)}\|\nabla \uu\|_{L^{3}(\Omega)}+\|\rho\|_{L^{\infty}(\Omega)}\|f\|_{L^{2}(\Omega)}+\|\nabla p\|_{L^{2}(\Omega)}).
\end{align*}
Using Sobolev and H\"older inequalities, we get
\begin{align*}
\|\nabla^{2}\uu\|_{L^{2}(\Omega)}&\leq C(\delta^{-1})\big(\|\sqrt{\rho}u_t\|_{L^2(\Omega)}^2+\|\rho\|_{L^\infty(\Omega)}+\|\nabla p\|_{L^2(\Omega}\big)+C \|\rho\|_{L^{\infty}(\Omega)}\|\nabla \uu\|_{L^{2}(\Omega)}^{3/2}\|\nabla \uu\|_{H^{1}(\Omega)}^{1/2}.
\end{align*}
Using Young's inequality, we deduce 
\begin{equation}\label{ch-22}
\|\nabla \uu(t)\|_{H^{1}(\Omega)}=(\|\nabla \uu\|_{L^{2}(\Omega)}^2+\|\nabla^{2} \uu\|_{L^{2}(\Omega)}^{2})^{1/2}\leq C(\delta^{-1})\, \psi(t)^{7/2}.
\end{equation}
We come back now to estimate the remaining integrals on the right-hand side of \eqref{19ch}. We have
\begin{align*}
&\int_{0}^{t}\int_{\Omega}\big(\rho |\uu|^2|\nabla \uu|^{2}+p|\uu||\nabla \divv \uu|+p|\divv \uu|^{2}\big)\,ds\\
&\leq C\int_{0}^{t}\big(\|\rho\|_{L^{\infty}(\Omega)}\|\uu\|_{L^{6}(\Omega)}^{2}\|\nabla \uu\|_{L^{3}(\Omega)}^{2}+\|p\|_{L^{3}(\Omega)}\|\uu\|_{L^{6}(\Omega)}\|\nabla \divv \uu\|_{L^{2}(\Omega)}\\
&\qquad\quad+\|p\|_{L^{3}(\Omega)}\|\divv \uu\|_{L^{2}(\Omega)}\|\divv \uu\|_{L^{6}(\Omega)}\big)\,ds\\
&\leq C\int_{0}^{t}\big(\|\rho\|_{L^{\infty}(\Omega)}\|\nabla \uu\|_{L^{2}(\Omega)}^{3}\|\nabla \uu\|_{H^{1}(\Omega)}+\|p\|_{L^{3}(\Omega)}\|\nabla \uu\|_{L^{2}(\Omega)}\|\nabla \uu\|_{H^{1}(\Omega)}\big)\,ds\\
&\leq C\int_0^t \psi(s)^6 \,ds.
\end{align*}
Substituting this estimate and estimate \eqref{ch-22} into \eqref{19ch}, we conclude that
\begin{equation}\label{23ch}
\|\nabla \uu(t)\|_{L^{2}(\Omega)}^{2}\leq C(\delta^{-1})+C(\delta^{-1}) \int_{0}^{t}\psi(s)^{6}\,ds.
\end{equation}
\hfill \break
{\bf Step 3}. \textit{Estimate for $\|\sqrt{\rho} \uu_t\|_{L^{2}(\Omega)}$}.
Let us begin by proving the following lemma which will be useful in the sequel.
\begin{lemma}\label{BS_lem2} For $\uu$ sufficiently regular, we have
\begin{align*}
\mathbb{J}:=-\int^t_0\int_{\Omega}\uu_t\cdot{\rm div}\Big( B^{\frac{q-2}{2}}\D(\uu)\Big)_t
\,dx\,dt\geq0\quad \mbox{for all}\quad q\geq 1.
\end{align*}
\begin{proof}
Through symmetry and integral by parts allow us to write
\begin{align*}
\mathbb{J}&=\int^t_0\int_{\Omega}\partial_j \uu_t\cdot\big( B^{\frac{q-2}{2}}  D_{ij}(\uu) \big)_t\,dx\,dt\\
&=\int^t_0\int_{\Omega}B^{\frac{q-2}{2}}\del_{j}\uu_{t}:\D_{ij}(\uu_t)\,dx\,dt+\dfrac{q-2}{2}\int^t_0\int_{\Omega} B^{\frac{q-4}{2}}(|\D_{ij}(\uu)|^{2})_{t} \D_{ij}(\uu):\del_{j}\uu_{t}\,dx\,dt\\
&=\int^t_0\int_{\Omega} B^{\frac{q-2}{2}}\del_{j}\uu_{t}:\D_{ij}(\uu_{t})\,dx\,dt+(q-2)\int^t_0\int_{\Omega}B^{\frac{q-4}{2}}|\D_{ij}(\uu)|^{2} \D_{ij}(\uu_t):\del_{j}\uu_t\,dx\,dt\\
&\geq \int^t_0\int_{\Omega} B^{\frac{q-4}{2}}\big((q-1)|\D(\uu)|^2+\delta^2  \big)|\D_{ij}(\uu_t)|^2\,dx\geq 0,
\end{align*}
for all $q\geq 1.$ The proof of Lemma \ref{BS_lem2} is now finished.
\end{proof}
\end{lemma}
We start by writing the momentum equation in a non-conservative form
\begin{equation}\label{no_cons-mom}
\rho \del_t\uu +(\rho \uu\cdot \nabla) \uu -\mu \Delta \uu -(\lambda+\mu)\nabla \divv \uu -\tau^* \divv\big(B^{\frac{q-2}{2}}\D(\uu)\big)+\nabla p=\rho f.
\end{equation}
By differentiating \eqref{no_cons-mom} with respect to time, we get
\begin{align*}
\rho \uu_{tt}&+\rho \uu\cdot\nabla \uu_t-\mu\Delta \uu_t-(\lambda+\mu)\nabla \divv \uu_t-\tau^*{\rm div}\big(B^{\frac{q-2}{2}}\D(\uu)\big)_t+\nabla p_t\\
&=(\rho f)_t-\rho_t(\uu_t+\uu\cdot \nabla \uu)-\rho \uu_t\cdot\nabla \uu.
\end{align*}
Taking the scalar product of the above equation by $\uu_t$ and
integrate the resulting over $\Omega$, one obtains after integration by parts
\begin{equation}\label{25-ch}
\begin{split}
\frac{1}{2}\dfrac{d}{dt}&\intO\rho|\uu_t|^2\,dx
+\mu\intO|\nabla \uu_t|^2\,dx +(\lambda+\mu)\intO |\divv \uu_t|^2\, dx+ \tau^*\mathbb{J}-\intO  p_t \divv \uu_t\,dx \\
&=-\intO \rho_t \big(\uu_t+\uu\cdot \nabla \uu-f\big)\cdot\uu_t\,dx -\intO(\rho \uu_t\cdot\nabla \uu)\cdot \uu_t\,dx+\intO\rho f_t\cdot \uu_t\, dx  
\\
& = - \intO\rho \uu\cdot\nabla((\uu_t+\uu\cdot \nabla \uu-f)\cdot \uu_t)\,dx-\intO(\rho \uu_t\cdot\nabla \uu)\cdot \uu_t\,dx+\intO\rho f_t\cdot \uu_t\, dx.
\end{split}
\end{equation}
Using equation \eqref{2.8-d}, we can write
\begin{align*}
-\intO p_t\divv \uu_t\,dx &= \intO (\nabla p\cdot \uu +\gamma p\divv \uu)\divv \uu_t\,dx\\
&=\intO (\nabla p\cdot \uu )\divv \uu_t\,dx+\dfrac{\gamma}{2}\Dt\intO p(\divv\uu)^2\,dx-\dfrac{\gamma}{2}\Dt \intO p_t (\divv \uu)^2\,dx\\
& = \frac{d}{dt}\intO\frac{\gamma}{2}p(\divv \uu)^2dx+\intO\nabla p\cdot (\uu\divv \uu_t)\,dx\\
&\quad +\frac{\gamma}{2}\Big(\intO-p \uu\cdot\nabla(\divv \uu)^2\,dx+(\gamma-1)\intO p(\divv \uu)^3\,dx\Big).
\end{align*}
Substituting this identity into \eqref{25-ch}, we obtain 
\begin{equation}\label{26ch}
\begin{split}
\frac{d}{dt}\intO &\left(\frac{1}{2}\rho|\uu_t|^2+\frac{\gamma}{2}p (\divv \uu)^2 \right)\,dx+\mu \intO |\nabla \uu_t|^2\,dx+(\lambda+\mu)\intO|\divv \uu_t|^2\,dx+\tau^*\mathbb{J} \\
& \le  C\intO\Big(|\rho| |\uu||\uu_t||\nabla \uu_t|+|\rho| |\uu||\uu_t|| \nabla \uu|^2+ |\rho| |\uu|^2|\uu_t||\nabla^2 \uu|+|\rho| |\uu|^2|\nabla \uu|| \nabla \uu_t|\\
&\quad \quad +|\rho| |\uu_t|^2|\nabla \uu|+ |\nabla p||\uu||\nabla \uu_t|+|p||\uu||\nabla \uu||\nabla^2 \uu|+|p||\nabla \uu|^3\\
&\quad \quad +|\rho||\uu||\uu_t||\nabla f|+|\rho||\uu||f||\nabla \uu_t|+|\rho||\uu_t||f_t|\Big)\,dx:=\sum_{k=1}^{11} I_k. \end{split}
\end{equation}
We shall estimate each term $I_k$ by making extensive use of Sobolev inequality, H\"older inequality and estimates of Step 2. This procedure is extremely  similar to \cite{Kim-SNS}, but we prefer performing the calculus in the Appendix Section \ref{Appen-sec} for the sake of completeness.  Substituting all the estimates established in  Section \ref{Appen-sec}  into \eqref{26ch}, we get
\begin{equation}
\dfrac{1}{2}\Dt \intO \Big(\rho|\uu_t|^{2}+\gamma p(\divv \uu)^{2}\Big)\,dx+ \intO |\nabla \uu_t|^{2}\,dx\leq C\psi^{11}(t)+C(\|f\|_{H^{1}(\Omega)}^{2}+\|f_t\|_{L^{2}(\Omega)}^{2}).
\end{equation}
Thus integrating over $(\tau,t)\subset\subset (0,T),$ we also obtain
\begin{equation}
\begin{split}\label{27-ch}
\dfrac{1}{2}\intO &\rho |\uu_t|^{2}(t)\,dx+ \int_{\tau}^{t}\intO |\nabla \uu_t|^{2}\,dx\,ds\\
&\leq C+\dfrac{1}{2}\intO \Big(\rho |\uu_t|^{2}(\tau)+\gamma p (\divv \uu)^{2}(\tau)\Big)\,dx+C\int_{\tau}^{t} \psi(s)^{11}\,ds.
\end{split}
\end{equation}
In the meantime, we can test the momentum equation by $\uu_t$ to obtain
\begin{equation*}
\intO \rho |\uu_t|^{2}\,dx\leq 2 \intO\Big(|\rho||\uu|^{2}|\nabla \uu|^{2}+|\rho||f|^{2}+|\rho^{-1}||\divv\big(2\mu \D(\uu)+\lambda\divv \uu\,\mathbb{I}+\tau^* B^{\frac{q-2}{2}}\D(\uu)\big)-\nabla p|^{2}\Big)\,dx.
\end{equation*} 
Therefore, letting $\tau\rightarrow 0^+$ in inequality \eqref{27-ch}, we conclude that
\begin{equation}\label{ch-28}
\|\sqrt{\rho}\uu_{t}(t)\|_{L^{2}(\Omega)}^{2}+\int_{0}^{\tau}\|\nabla \uu_{t}(s)\|_{L^{2}(\Omega)}^{2}\,ds\leq C +C(\rho_0,\uu_0)+ C\int_{0}^{t}\psi(s)^{11}\,ds
\end{equation}
where 
\begin{equation*}
C(\rho_0, \uu_0)=\intO \rho_{0}^{-1}\big|\divv\big(2\mu \D(\uu_0)+\lambda\divv \uu_0\,\mathbb{I}+\tau^* (|\D(\uu_0)|^2 +\delta^2)^{\frac{q-2}{2}}\D(\uu_0)\big)-\nabla p_0\big|^{2}\,dx=\|g\|_{L^{2}(\Omega)}^2
\end{equation*}
which is finite due to the compatibility condition \eqref{comp-13}.

\hfill \break
{{\bf Step 4}. \it Estimate for $\|p\|_{W^{1,6}(\Omega)}$.}
Recall that from the continuity equation, we have
\begin{equation}\label{cont-pressure}
p_t+\divv(p\uu)+(\gamma-1)p\divv \uu=0.
\end{equation}
Differentiating it with respect to $x_k$ yields
\begin{equation*}
(p_{x_k})_t+\uu\cdot \nabla p_{x_k}+\nabla p\cdot \uu_{x_k}+\gamma p_{x_k}\divv \uu +\gamma p\divv \uu_{x_k}=0.
\end{equation*}
Now, we multiply the above equation by $p_{x_k}|p_{x_k}|^4$ and we integrate over $\Omega.$ After summing  over $k$, we get
\begin{equation*}
\begin{split}
\frac{d}{dt}\|\nabla p\|^{6}_{L^{6}(\Omega)}
 &\leq  C\intO |\nabla \uu||\nabla p|^6\,dx+\intO p|\nabla \divv \uu||\nabla p|^{5}\,dx\\
& \leq  C \big( \|\nabla \uu\|_{L^{\infty}(\Omega)}\|\nabla p\|_{L^6 (\Omega)}^6+\|p\|_{L^\infty(\Omega)} \|\nabla \divv \uu\|_{L^6(\Omega)}\|\nabla p\|_{L^6 (\Omega)}^{5}\big)
\end{split}
\end{equation*}
which implies that
\begin{equation*}
\frac{d}{dt}\|\nabla p\|_{L^6(\Omega)}
 \le   C( \|\nabla \uu\|_{L^\infty(\Omega)}\|\nabla p\|_{L^6 (\Omega)}+\|p\|_{L^\infty (\Omega)} \|\nabla^2 \uu\|_{L^6 (\Omega)}).
\end{equation*} 
Similarly, we can derive a bound on $\|p\|_{L^{6}(\Omega)}$ by multiplying \eqref{cont-pressure} by $p$ and integrate over $\Omega$. Thus, we get
\begin{equation*}
\Dt\|p\|_{W^{1,6}(\Omega)}\leq C\| \uu\|_{W^{2,6}(\Omega)}\|p\|_{W^{1,6}(\Omega)}.
\end{equation*}
Therefore, by virtue of Gronwall's inequality we deduce 
\begin{equation}\label{ch-29}
\|p(t)\|_{W^{1,6}(\Omega)}\leq C \;{\rm exp}\Big( \int_0^t \| \uu\|_{W^{2,6}(\Omega)}\,ds\Big).
\end{equation}
In the same manner, we can also show that
\begin{equation*}
\|\rho(t)\|_{W^{1,6}(\Omega)}\leq C \;{\rm exp}\Big(\int_0^t \| \uu\|_{W^{2,6}(\Omega)}\,ds\Big).
\end{equation*}
In order to complete this step, one has to show the $L^p-$estimates on the elliptic operator associated to the system based on the results that we showed in Proposition \ref{prop-reg}. Indeed, we write the momentum equation as follows
\begin{equation}\label{lp-prob}
-\mathcal{A}(\uu^*,\D)\uu=-\rho \uu_t-\rho (\uu\cdot\nabla \uu)+\rho f-\nabla p+(\mathcal{A}(\uu,\D)\uu-\mathcal{A}(\uu^*,\D)\uu)
\end{equation}
where $\uu^*$ is a reference solution. Investigating \eqref{lp-prob} as a quasi-linear elliptic equation, then according to the last statement in Proposition \ref{prop-reg}, we know that the linearized operator still yield maximal $L^p-$regularity. Nevertheless, by definition of the operator $\mathcal{A}$ in \eqref{defA}, we have
\begin{align*}
\mathcal{A}(\uu,\D)\uu-\mathcal{A}(\uu^*,\D)\uu=&(\beta(B)-\beta(B^*))(\Delta \uu+\nabla \divv \uu) \\
&+4\underset{j,k,l=1}{\overset{3}{\sum}}(\beta^\prime(B)\D_{ik}(\uu)\D_{jl}(\uu)-\beta^\prime(B^*)\D_{ik}(\uu^*)\D_{jl}(\uu^*))\del_k \del_l\uu_j,
\end{align*}
where we denoted by $B^*=|\D(\uu^*)|^2+\delta^2$, and remember that $B$ and $\beta(\cdot)$ were defined in \eqref{def-gamma}. Therefore, we have
\begin{equation}\label{difA-A*-est}
\begin{split}
\|& \mathcal{A}(\uu,\D)\uu-\mathcal{A}(\uu^*,\D)\uu \|_{L^p(\Omega)}\\
&\leq C \| |\D(\uu)|^2 - |\D(\uu^*)|^2\|_{L^{\infty}(\Omega)} \| \uu \|_{W^{2,p}(\Omega)}\\
&+ C\underset{k,l=1}{\overset{3}{\sum}} \| \big|\D_{ik}(\uu)\D_{jl}(\uu)-\D_{ik}(\uu^*)\D_{jl}(\uu^*)\|_{L^{\infty}(\Omega)}\| \uu \|_{W^{2,p}(\Omega)}\\
&+4 C\underset{k,l=1}{\overset{3}{\sum}} \| |(\D(\uu)|^2 - |\D(\uu^*)|^2)\D_{ik}(\uu^*)\D_{jl}(\uu^*)\|_{L^{\infty}(\Omega)}\| \uu \|_{W^{2,p}(\Omega)}\\
&\leq C  \| |\D(\uu)| - |\D(\uu^*)|\|_{L^{\infty}(\Omega)} \| |\D(\uu)| + |\D(\uu^*)|\|_{L^{\infty}(\Omega)}\| \uu \|_{W^{2,p}(\Omega)}\\
&+C\underset{k,l=1}{\overset{3}{\sum}}\big(|\D_{ik}(\uu)|_{L^\infty(\Omega)}\| |\D_{jl}(\uu)| - |\D_{jl}(\uu^*)|\|_{L^{\infty}(\Omega)}+ |\D_{jl}(\uu)|_{L^\infty(\Omega)}\| |\D_{ik}(\uu)| - |\D_{ik}(\uu^*)|\|_{L^{\infty}(\Omega)}\big) \| \uu \|_{W^{2,p}(\Omega)}\\
&+\| |(\D(\uu)| - |\D(\uu^*)|)\|_{L^{\infty}(\Omega)} \|\D_{ik}(\uu^*)\D_{jl}(\uu^*)\|_{L^{\infty}(\Omega)}\| |\D(\uu)| + |\D(\uu^*)|\|_{L^{\infty}(\Omega)}\| \uu \|_{W^{2,p}(\Omega)}\\
&\leq C \| \nabla(\uu-\uu^*)\|_{L^{\infty}(\Omega)}\| \uu\|_{W^{2,p}(\Omega)}.
\end{split}
\end{equation} 
Thus, if $\uu$ represents a small perturbation of $\uu^*$, we deduce from \eqref{lp-prob} the following estimate
\begin{equation}
\|\uu \|_{W^{2,p}(\Omega)}\leq C(\delta^{-1}) \|-\rho \uu_t-\rho (\uu\cdot\nabla \uu)+\rho f -\nabla p\|_{L^{p}(\Omega)}+\varepsilon\|\uu\|_{W^{2,p}(\Omega)},
\end{equation}
and whence\footnote{The estimates \eqref{Lp-ch}-\eqref{ch30}-\eqref{31-ch}-\eqref{ch-32}-\eqref{33-ch} hold uniformly with respect to $\delta$ in one dimensional space due to the first statement in Proposition \ref{prop-reg}.}
\begin{equation}\label{Lp-ch}
\begin{split}
\int_0^t \| \uu\|_{W^{2,6}(\Omega)}&\leq C(\delta^{-1}) \int_0^t \big(\|\rho \uu_t \|_{L^6(\Omega)}+\|\rho \uu\cdot \nabla \uu \|_{L^6(\Omega)}+\|\rho f \|_{L^6(\Omega)}+\|\nabla p \|_{L^6(\Omega)}  \big)ds\\
&\leq C(\delta^{-1}) \int_0^t \big(\|\rho \|_{L^{\infty}(\Omega)}\|\nabla \uu_t\|_{L^2(\Omega)}+\|\rho \|_{L^{\infty}(\Omega)} \|\nabla \uu \|^2_{L^{6}(\Omega)}+\|\rho \|_{L^{\infty}(\Omega)}\|f \|_{H^{1}(\Omega)}+\|\nabla p \|_{L^{6}(\Omega)})\\
&\leq  C(\delta^{-1})+C(\delta^{-1})\int_0^t (\psi(s)^{8}+\|\nabla \uu_{t}(s)|_{L^{2}(\Omega)}^{2})\,ds.
\end{split}
\end{equation}
Now, using estimates \eqref{ch-22},  \eqref{ch-28}, and \eqref{Lp-ch}, we infer from \eqref{ch-29} the following
\begin{equation}\label{ch30}
\|p(t)\|_{W^{1,6}(\Omega)}\leq C(\delta^{-1}) \;{\rm exp}\Big(c(\rho_0,\uu_0)+c\int_0^t \psi(s)^{11}\,ds\Big).
\end{equation}
\hfill \break
{\bf Step 5}. {\it Final estimate.}
Gathering  \eqref{ch-22}, \eqref{23ch}, \eqref{ch-28} and  \eqref{ch-29}, we deduce that
\begin{equation}\label{31-ch}
\psi(t)+\|\nabla \uu(t)\|_{H^{1}(\Omega)}+\int_0^t \|\nabla \uu_{t}(s)\|_{L^{2}(\Omega)}^{2}\,ds\leq  C(\delta^{-1})\;{\rm exp}\Big(\int_0^t \psi^{11}(s)\,ds\Big)
\end{equation}
where $C$ is a constant which may depend on $c(\rho_0,\uu_0)=\|g\|_{L^{2}(\Omega)}^{2}.$ From this estimate \eqref{31-ch}, we can find a small time $T^*>0$ and a constant $C>0$ depending on $a, \gamma, \mu,\lambda, T, \|g\|_{L^{2}(\Omega)}$ and the norms of the data $(\rho_0, \uu_0, f),$ such that
\begin{equation}\label{ch-32}
\underset{0\leq t\leq T^*}{\sup}\big(\|\nabla \uu\|_{H^{1}(\Omega)}+\|\sqrt{\rho} \uu_t\|_{L^{2}(\Omega)}+\|p\|_{W^{1,6}(\Omega)}\big)+\int_0^{T^*} \|\nabla \uu_t\|_{L^{2}(\Omega)}^{2}\,dt\leq  C(\delta^{-1}).
\end{equation}
In the meantime, using Sobolev inequality and estimate \eqref{ch-22}, we deduce from \eqref{ch-32} the following
\begin{equation}\label{33-ch}
\underset{0\leq t\leq T^*}{\sup}\big(\|\uu\|_{L^{6}(\Omega)}+\int_0^{T^*} (\|\uu(t)\|_{W^{2,6}(\Omega)}^{2}+\|\uu_{t}\|_{L^{6}(\Omega)}^{2}\big)\,dt\leq  C(\delta^{-1}).
\end{equation} 
\subsection{Construction of approximate solutions}
Have established the necessary a priori estimates in the previous subsection, we can now proceed to the construction of approximate solutions following the ideas developed for compressible Navier-Stokes equations (case of Newtonian fluids) that can be found in several papers and books, see for instance \cite{p.lions1, Fe-No-book, Fe-No-cons}. Precisely, we construct approximate solutions in an appropriate finite-dimensional space by a fixed-point argument combined with the characteristics method  and the Faeodo-Galerkin method. The existence of solutions will be extended to the complete space with the help of the a priori estimates established in the previous subsection. Since this procedure is mostly classical now, then we will omit some details here.
Indeed, we take our basic function spaces as $X= H^{2}(\Omega)$ and its finite-dimensional subspaces as
\begin{equation*}
X_{m}={\rm span}\{\phi_{1},\cdots,\phi_m\}\subset X\cap C^{2}(\overline{\Omega}),
\end{equation*}
where $\phi_i$ be an orthonormal basis of $L^2(\Omega)$ which is also an orthogonal basis of $H^2(\Omega)$. Such a basis surely exists since $C^{2}(\overline{\Omega})\cap X$ is dense in $X.$ Let $\rho_0, \uu_0$ and $f$ be the data satisfying the hypothesis of Theorem \ref{theo1}. We assume for the moment that $\rho_0 \in C^1(\overline{\Omega})$ satisfying $\rho_0(x) \geq \varepsilon$ for $x\in \Omega$ for some $\varepsilon>0.$ Then, we are looking for  $(\rho_m, \uu_m)$ solution to the following system
\begin{equation}\label{rh.a-BS}
\del_t \rho_m+\divv(\rho_m \uu_m)=0 ,
\end{equation}
\begin{equation}\label{um.ap-BS}
\intO \big(\rho_m \del_t\uu_m+ (\rho_m \uu_m\cdot\nabla )\uu_m  -\divv(\Ss_\delta^m) +\nabla p_m\big)\cdot v\,dx =\intO \rho_m f\cdot v\,dx\qquad v\in X_m,
\end{equation}
\begin{equation}\label{intital-um}
\uu_m(0)=\uu_0^m=\underset{k=1}{\overset{m}{\sum}}\langle \uu_0, \phi^k\rangle_{L^2(\Omega)}\,\phi^k  \quad \mbox{and}\quad \rho_m(0)=\rho_0.
\end{equation}
Following \cite{Fe-No-cons}, we can prove that there exist $(\rho_m, \uu_m)$ solution of the approximate system \eqref{rh.a-BS}-\eqref{um.ap-BS}-\eqref{intital-um}. Indeed, for any given $\uu_m\in C^1([0,T); X_m)$, then by the classical theory of transport equation, there exists a classical solution $\rho_m(t,x)\in C^1([0,T); X_m)$ to \eqref{rh.a-BS} with initial data $\rho_0(x)\geq \varepsilon>0$. Moreover, we can derive the following a priori bound on $\rho_m$ 
\begin{align*}
\underset{x\in \Omega}{\inf}\; \rho_0(x)\exp \Big(-\int_0^t \|\divv \uu_m\|_{L^\infty(\Omega)}\,ds\Big)\leq \rho_m(t,x)\leq \underset{x\in \Omega}{\sup}\; \rho_0(x)\exp \Big(\int_0^t \|\divv \uu_m\|_{L^\infty(\Omega)}\,ds\Big).
\end{align*}
Since we assumed that $\rho_0(x)>0$ for all $x\in \Omega,$ then we deduce that
$$\rho_m(t,x)>0 \quad \mbox{for all}\quad (t,x)\in [0,T)\times \Omega.$$
We introduce the operator 
\begin{align*}
T:C^1([0,T]; X_m) &\go C^1([0,T]; C^{1}(\Omega))\\
\uu_m &\go T(\uu_m)=\rho_m.
\end{align*}
Since the equation for $\rho_m$ is linear, $T$ is Lipschitz continuous in the following sense:
\begin{equation}\label{Ju-3.3}
\|T(v_1)-T(v_2)\|_{C^1([0,T]; C^{1}(\Omega))}\leq  c\|v_1-v_2\|_{C([0,T]; L^{2}(\Omega)}).
\end{equation} 
Next, we wish to solve the momentum equation \eqref{um.ap-BS}. For that we will apply Banach's fixed-point theorem to prove the local-in-time existence of solution. Indeed, we assume that $\rho_m=T(\uu_m)$ is given and we introduce a family of operator defined as follow: for a given function $\varrho \in L^1(\Omega)$ with $\varrho\geq \underline{\vrho}>0:$
\begin{align*}
M[\rho]: X_m \go X_m^* \quad \quad \langle M[\rho] \uu_m, w\rangle=\intO \vrho v\cdot w\,dx\quad v, w\in X_m.
\end{align*}  
These operators are symmetric and positive definite with the smallest eigenvalue
\begin{align*}
\underset{\|w\|_{L^2(\Omega)}=1}{\inf} \langle M[\rho]w,w\rangle=\underset{\|w\|_{L^2(\Omega)}=1}{\inf} \intO \varrho |w|^2\,dx\geq \underset{x\in \Omega}{\inf} \varrho(x)\geq \underline{\varrho}.
\end{align*}
Then, since $X_m$ is finite dimensional, the operators are invertible with 
$$\|M^{-1} [\rho]\|_{\mathcal{L}(X_m^*,X_m)}\leq \underline{\vrho}^{-1},$$
where we denote by $\mathcal{L}(X_m^*,X_m)$ the set of bounded linear mappings from $X_m^*$ to $X_m$. Moreover, see \cite[Chap. 7.3.3]{fe-book}, $M^{-1}$ is Lipschitz continuous in the sense
\begin{equation}\label{3.5-Ju}
\|M^{-1}[\vrho_1]-M^{-1}[\vrho_2]\|_{\mathcal{L}(X_m^*,X_m)} \leq C \|\vrho_1-\vrho_2\|_{L^1(\Omega)}
\end{equation} 
for $\vrho_1, \vrho_2\in L^{1}(\Omega)$ such that $\vrho_1, \vrho_2 \geq \underline{\vrho}>0.$
We are looking for $\uu_m$ solution of the following nonlinear integral equation
$$\uu_m(t)=M^{-1}[T(\uu_m)(t)]\Big(M[\rho_0]\uu_0^m +\int_0^t N[T(\uu_m), \uu_m(s)]\,ds\Big) \quad \mbox{in} \quad X_m$$
where we denote by
\begin{align*}
N(T(\uu_m), \uu_m)=\rho_m f-\divv(\rho_m \uu_m\otimes \uu_m)+\divv(\Ss_\delta^m)-a\nabla \rho^\gamma_m.
\end{align*}
By virtue of estimates \eqref{Ju-3.3} and \eqref{3.5-Ju} for the operators $T$ and $M^{-1}$, we can solve the above nonlinear equation by evoking the fixed-point theorem of Banach on a short time interval $[0, T^*]$ with $T^*<T$ in the space $C([0,T^*], X_m)$.

Up to now, we have proved that there exists $(\rho_m, \uu_m)$ solution to \eqref{rh.a-BS}-\eqref{um.ap-BS}-\eqref{intital-um} with the following properties
  $$\rho_m\in C^{1}([0,T^*); X_m ) \quad \quad \uu_{m}\in C^{1}([0,T^*); X_m).$$
Now, we will show that these approximate solutions satisfy the following uniform estimate (with respect to $m$ and $\varepsilon$) analogue to \eqref{31-ch}
\begin{equation}\label{ch-37}
\begin{split}
\psi_{m}(t)+\|\nabla \uu_m\|_{H^{1}(\Omega)}+\int_0^t \|\nabla (\uu_{m})_{t}\|_{L^{2}(\Omega)}^{2}\,ds  \leq  C \,{\rm exp} \Big(C(\rho^0,\uu^0_m)+C\int_0^t \psi_{m}(t)^{11}\,ds\Big),
\end{split}
\end{equation}
where we have denoted by
\begin{equation}\label{14-ch}
\psi_{m}(t)=1+\|\nabla \uu_m\|_{L^{2}(\Omega)}^{2}+||\sqrt{\rho_m}(\uu_m)_t||_{L^{2}(\Omega)}^{2}+||p_m(t)||_{W^{1,q}(\Omega)}^{2}.
\end{equation}
We start by taking $v=u_m(t)$ in equation \eqref{um.ap-BS}, integrate over $(0,t)$ and apply Gronwall's inequality. Then, we obtain the analogue of estimate \eqref{eq15-c}
\begin{equation}\label{ch-28-dis}
\intO \big( \rho_m(t) |\uu_m(t) |^2 +p_{m}(t)\big)\,dx+\int_{0}^{t}\intO |\nabla \uu_m(s)|^{2}\,dx\,ds\leq C.
\end{equation}
This  implies immediately the analogue of estimate \eqref{ch-17} 
\begin{equation}\label{ch-39}
||\rho_{m}(t)||_{L^{q}(\Omega)}+||p_m(t)||_{L^{q}(\Omega)}\leq C \psi_m(t),\;\;\;{\rm for\;\; all}\;\;\;1\leq q\leq \infty.
\end{equation}
At any time $t,$ taking $v=(\uu_{m})_t$ in equation \eqref{um.ap-BS} yields to
\begin{equation*}
\dfrac{1}{2}\intO \rho_m|\del_t \uu_m |^{2}\,dx+\dfrac{\mu}{2}\dfrac{d}{dt} \intO|\nabla \uu_m|^{2}\,dx
\leq \intO \big(\rho_m|f|^{2}+\rho_m|\uu_m|^2|\nabla \uu_m|^2\big)\,dx +\intO p_m \del_t\divv \uu_{m} \,dx
\end{equation*}
As did in the last subsection, one may establish the $H^2-$estimate on the approximate solution $(\uu_m)$ by taking $v=-\Delta \uu_m$ and follows the same lines as in the second statement of Proposition \ref{prop-reg}.  We deduce 
\begin{equation*}
 \|\nabla^{2}\uu_{m}\|_{L^{2}(\Omega)}\leq C\big( \|\del_{t}(\rho_m \uu_m)\|_{L^{2}(\Omega)}+\|\divv(\rho_m \uu_m\otimes \uu_m)\|_{L^{2}(\Omega)}+\|\nabla p_m\|_{L^{2}(\Omega)}+\| \rho_m f\|_{L^{2}(\Omega)}\big).
\end{equation*}
The analogue of the above estimate, i.e. the $L^p-$estimates could be established as follows. First, we may rewrite \eqref{um.ap-BS} as follows
\begin{align}\label{A-A*}
\begin{split}
\intO -\mathcal{A}(\uu_0^m, \D)\uu \cdot v\,dx =&\intO [\rho_m f- \rho_m \del_t\uu_m - (\rho_m \uu_m\cdot\nabla )\uu_m  -\nabla p_m]\cdot v\,dx\\
&+\intO (\mathcal{A}(\uu, \D)\uu-\mathcal{A}(\uu_0^m, \D)\uu) \cdot v\,dx,
\end{split}
\end{align}
where we remember that $\mathcal{A}$ is defined in \eqref{defA}, and $\mathcal{A}(\uu_0^m, \D)$ is the linearized operator at $\uu_0^m$. Second, we take $v=-|\mathcal{A}(\uu_0^m, \D)|^{p-2}\mathcal{A}(\uu_0^m, \D)$. For the last integral on the right-hand side of equation \eqref{A-A*}, we proceed as in \eqref{lp-prob}-\eqref{difA-A*-est}. At the end, we conclude with the following estimate, which holds uniformly with respect to $m,$ but not uniformly with respect to\footnote{In one dimensional space, this estimate holds uniformly with respect to $\delta$ thanks to Proposition \ref{prop-reg}.} $\delta$
\[
 \|\nabla^{2}\uu_{m}\|_{L^{p}(\Omega)}\leq C(\delta^{-1})\big( \|\del_{t}(\rho_m \uu_m)\|_{L^{p}(\Omega)}+\|\divv(\rho \uu_m\otimes \uu_m)\|_{L^{p}(\Omega)}+\|\nabla p_m\|_{L^{p}(\Omega)}+\|\rho_m f\|_{L^{p}(\Omega)}\big).
\]
In the meantime, we can derive the analogues of \eqref{ch-22} and \eqref{23ch}
\begin{equation}\label{ch40}
\|\nabla \uu_m\|_{H^{1}(\Omega)}\leq C \psi_{m}(t)^{7/2}\quad \mbox{and}\quad \|\nabla \uu_m(t)\|_{L^{2}(\Omega)}^{2}\leq C+C \int_{0}^{t}\psi_{m}(s)^{6}\,dx.
\end{equation}
In addition, differentiating equation \eqref{um.ap-BS} with respect to time and taking $v=(\uu_m)_t$ and apply the same procedure as in the previous subsection, we can infer that
\begin{equation*}
\|\sqrt{\rho}\del_t \uu_{m}(t)\|_{L^{2}(\Omega)}^{2}+\int_{0}^{t}\|\del_t\nabla \uu_{m} (s)\|_{L^{2}(\Omega)}^{2}\,ds\leq C +C(\rho_0,\uu_0^m)+ C\int_{0}^{t}\psi_m(s)^{11}\,ds.
\end{equation*}
Similarly as in estimate \eqref{ch30}, we have also
\begin{equation}\label{ch42}
\|p_m(t)\|_{W^{1,6}(\Omega)}\leq C \;{\rm exp}\Big(C(\rho_0,\uu_0^m)+C\int_0^t \psi_m(s)^{11}\,ds\Big).
\end{equation}
Combining \eqref{ch40} and \eqref{ch42}, we deduce  \eqref{ch-37}. On the other hand, we conclude from \eqref{ch-37} that
\begin{equation}\label{ch43}
\varphi_{m}(t)\leq C \mathcal{C}(\rho_0, \uu^{m}_{0})+C\int_{0}^{t}\exp(C\varphi_{m}(s))\,ds,
\end{equation}
where $\varphi_{m}(t)=\log(\psi_{m}(t)/C).$ Note that since $\mathcal{C}(\rho_0,\uu_0)=\|g\|^2_{L^{2}(\Omega)}$ and 
$$|\mathcal{C}(\rho,\uu_0^m)-\mathcal{C}(\rho_0,\uu_0)|\leq \dfrac{C}{\varepsilon}\|\uu_0^m-\uu_0\|_{H^{2}(\Omega)}\rightarrow 0\;\;{\rm as}\;\;m\rightarrow \infty,$$
there exists a large integer $M=M(\varepsilon)>0$ such that $\mathcal{C}(\rho_0,\uu_0^m)\leq \|g\|_{L^{2}(\Omega)}^{2}+1$ for all $m\geq M$. Hence using the integral inequality \eqref{ch43}, we can infer that there exists a small time $T^*\in (0,T),$ independent of $\delta$ and $m,$ such that
\begin{equation*}
\underset{0\leq t\leq T^*}{\sup}\psi_{m}(t)\leq C\exp (C \mathcal{C}(\rho_0,\uu_0^m))\quad \mbox{for all} \quad m\geq M.
\end{equation*} 
In the meantime, using \eqref{ch-37} together with this uniform bound on the function $\psi_m$, we can also derive the analogues of \eqref{ch-32} and \eqref{33-ch} for each bound $m\geq M:$
\begin{equation}\label{44-ch}
\begin{split}
&\underset{0\leq t\leq T^*}{\sup}\big(\|\sqrt{\rho_m}\del_t \uu_m \|_{L^{2}(\Omega)}+\|\nabla \uu_m\|_{H^{1}(\Omega)}+\|p_m\|_{W^{1,6}(\Omega)}\big)\\
\quad &+\int_{0}^{T^{*}}\big(\|\uu_m\|_{W^{2,6}(\Omega)}^{2}+\|\del_t \nabla \uu_m \|_{L^{2}(\Omega)}^{2}+\|\del_t \uu_{m} \|_{L^{6}(\Omega)}^{2}\big)\,dt\leq C\exp (C \mathcal{C}(\rho_0,\uu_0^m)). 
\end{split}
\end{equation}
Since $\mathcal{C}(\rho_0,\uu_0^m)\leq \|g\|_{L^{2}(\Omega)}^{2}+1$ for all $m\geq M,$ we can deduce from \eqref{44-ch} that the sequence $(\rho_m,\uu_m)$ converges, up to the extraction of subsequences, to some limit $(\rho,\uu)$ in the obvious weak sense, that  is 
\begin{equation*}
\begin{split}
&\rho_m \rightharpoonup^* \rho_\delta \;\;{\rm in}\;\; L^{\infty}(0,T^*; W^{1,6}(\Omega))\quad\;\;\uu_m\rightharpoonup^* \uu_\delta\;\;{\rm in}\;\; L^{\infty}(0,T^*; H^{1}(\Omega))\\
&\uu_m\rightharpoonup \uu_\delta\;\;{\rm in}\;\; L^{2}(0,T^*; W^{2,6}(\Omega))\quad \;\;(\uu_m)_t\rightharpoonup (\uu_\delta)_t \;\;{\rm in}\;\; L^{2}(0,T^*; H^{1}(\Omega))
\end{split}
\end{equation*}
These convergence properties are sufficient to pass to the limit in $m$. Let us just give some details about the convergence of the extra term that we have compared to \cite{Kim-SNS}. Indeed, denote by 
\begin{equation}\label{def-F}
F_\delta^q(A)=(|A|^2 +\delta^2)^{\frac{q-2}{2}} A \qquad \quad A\in \mathbb{R}^{d\times d}.
\end{equation}
Observe that due to \eqref{44-ch}, we know that
\[
|F_\delta^q (\D(\uu_m(x)))|\leq C \quad \mbox{for all}\quad x\in \Omega.
\]
On the other hand, we have
$$\D(\uu_m)\rightarrow \D(\uu_\delta) \quad\mbox{ a.e. in} \quad (0,T)\times \Omega.$$
Thus, by the continuity of $F_\delta$, we get
\[
F_\delta^q (\D(\uu_m)) \rightarrow F_\delta^q (\D(\uu_\delta))  \quad\mbox{ a.e. in} \quad (0,T)\times \Omega.
\]
Owing to the Vitali's convergence theorem, the passage to the limit in $m$ is now complete.  Furthermore, by the lower semi-continuity of norm, the solution satisfies the following estimate
\begin{equation}\label{45=ch}
\begin{split}
&\underset{0\leq t\leq T^*}{\sup} \big(\|\sqrt{\rho}\del_t \uu_\delta \|_{L^{2}(\Omega)}+\|\nabla \uu_\delta\|_{H^{1}(\Omega)}+\|p_\delta\|_{W^{1,6}(\Omega)}\big)\\
&\quad +\int_0^{T^*} \big(\|\uu_\delta\|_{W^{2,6}(\Omega)}+\|\del_t \nabla \uu_\delta \|_{L^{2}(\Omega)}^{2}+\|\del_t \uu_\delta \|_{L^{6}(\Omega)}^{2}\big)\,dt\leq C.
\end{split}
\end{equation}

So far, we have proved that there exist $(\rho_\delta,\uu_\delta)$ that solves \eqref{Power-Law}-\eqref{Sdelta-PL} with smooth initial data and initial density far from zero. We want now to extend the above existence's result to the case where the initial data satisfies the hypothesis of Theorem \ref{theo1}. Indeed, take $\rho_0^\delta, \uu_0^\delta$ and $f$ as in Theorem \ref{theo1}. For each $\varepsilon\in (0,1)$, we choose $\rho_0^\varepsilon\in C^1(\overline{\Omega})$ such that 
\[
\rho_0^\varepsilon(x)\geq \varepsilon \;\; \forall x\in \Omega \qquad \mbox{and}\quad \rho_0^\varepsilon \rightarrow \rho_0^\delta \quad \mbox{in}\quad W^{1,6}(\Omega).
\] 
Set $\uu_0^\varepsilon\in W^{2,6}$ be the unique solution to the following system 
\[
-\divv\Big(2\mu \D(\uu_{0}^\varepsilon)+\lambda\divv \uu_{0}^\varepsilon\,\mathbb{I}+\tau^* (|\D(\uu_{0}^\varepsilon)|^2 +\delta^2)^{\frac{q-2}{2}} \D(\uu_{0}^\varepsilon)\Big)=\sqrt{\rho_{0}^\varepsilon}g-a\nabla (\rho_{0}^\varepsilon)^\gamma.
\]
Such a solution surely exists due to Section \ref{Sec-NES}, Proposition \ref{prop-reg}. Finally, we denote by $(\rho^\varepsilon, \uu^\varepsilon)$ the local strong solution in $[0, T^*]$ to \eqref{Power-Law}-\eqref{Sdelta-PL} with initial data replaced by $(\rho_0^\varepsilon, \uu_0^\varepsilon)$. Notice that $T^*$ is independent of $\varepsilon$ due to the previous arguments.   Next, since $\uu^\varepsilon_0 \rightharpoonup \uu_0^\delta $ in $H^{2}(\Omega)$, and $C(\rho_0^\varepsilon, \uu_0^\varepsilon)=\| g\|_{L^2(\Omega)}^2<+\infty,$  and $(\rho^\varepsilon, \uu^\varepsilon)$ satisfies the uniform bound \eqref{45=ch}, the same arguments employed previously show that a subsequence of approximate solutions converges to a strong solution  to the original problem, which satisfies bound \eqref{45=ch} too. This completes the proof of Theorem \ref{theo1} except the uniqueness
assertion.

\subsection{Uniqueness of solution}
In this subsection we prove the uniqueness of solutions to system \eqref{Power-Law}-\eqref{Sdelta-PL}. Our proof is inspired of the method used in \cite{Des-SNS, Kim-SNS} in case of a Newtonian fluid, however, we shall use tools employed for incompressible non-Newtonian systems to deal with the nonlinear viscosity term.

Let us assume that there exist two solutions $(\rho, \uu)$ and $(\bar{\rho}, \bar{\uu})$ (we drop the index $\delta$ for simplicity) to system \eqref{Power-Law}-\eqref{Sdelta-PL} with the same initial data $(\rho_0, \uu_0).$ 
For the sake of simplicity, let us denote by
$$\vartheta=\rho -\bar{\rho}\quad z=\uu-\bar{\uu} \quad  \pi=p-\bar{p}.$$
We can easily verify that triple $(\vartheta,z,\pi)$ satisfies the following equation
\begin{equation}\label{eq-diff}
\begin{split}
\rho z_t+\rho \uu\cdot \nabla z-\divv(\Ss_\delta-\bar{\Ss}_\delta)
 &=\vartheta(f-\bar{\uu}_t-\bar{\uu}\cdot\nabla \bar{\uu})-\rho z\cdot\nabla \bar{\uu}-\nabla \pi\\
\quad &=\vartheta h-\rho z\cdot\nabla \bar{\uu}-\nabla \pi
\end{split}
\end{equation}
where $\Ss_\delta$ is defined in \eqref{Sdelta-PL} and $h:=f-\bar{\uu}_t-\bar{\uu}\cdot\nabla \bar{\uu}\in L^{2}(0,T; L^{6}(\Omega)).$
Now, multiplying equation \eqref{eq-diff} by $z$ and integrating over $\Omega$, we obtain
\begin{align}\label{diffe}
\begin{split}
\dfrac{1}{2}\Dt \intO &\rho |z|^2\,dx+\intO \Ss_\delta-\bar{\Ss}_\delta: \nabla  z \,dx\\
&\leq  \intO |\vartheta||h||z|\,dx +\intO \rho |z|^2|\nabla \bar{u}|\,dx+\intO \pi \divv z\,dx.
\end{split}
\end{align}
According to the definition of $\Ss_\delta$ given in \eqref{Sdelta-PL} we have
\begin{align*}
  \intO \Ss_\delta-\bar{\Ss}_\delta: \nabla  z \,dx = &\mu \intO |\nabla z|^2\,dx +(\lambda +\mu)\intO |\divv z|^2\,dx \\
 & + \tau^* \intO \big[(|\D(\uu)|^2 +\delta^2)^{\frac{p-2}{2}}\D(\uu)-(|\D(\bar{\uu})|^2 +\delta^2)^{\frac{p-2}{2}}\D(\bar{\uu})\big] : \nabla z\,dx.
\end{align*}
Let us treat the last integral in the above equation. Indeed, we define $\Phi: \mathbb{R}^{n\times n} \rightarrow \mathbb{R}^{+}$ as follows
\[
\Phi(A)=\frac{1}{q}(|A|^2+\delta^2)^{\frac{q}{2}} \quad \mbox{for }\quad A\in \mathbb{R}^{n\times n}.
\]
Therefore, for any two matrices $C, D$ in $\RR^{n\times n},$ we compute
\begin{align}\label{diff-CD}
\begin{split}
(|C_{ij}|^2+\delta^2)^{\frac{q-2}{2}}C_{ij}-(|D_{ij}|^2+\delta^2)^{\frac{q-2}{2}}D_{ij}&=\del_{ij}\Phi (C) -\del_{ij}\Phi(D)\\
&=\int_0^1 \dfrac{d}{ds}\Big(\del_{ij}\Phi(D+s(C-D))\Big)ds\\
&=\int_0^1 \underset{kl}{\sum} \del_{kl}\del_{ij}\Phi(D+s(C-D)) (C_{kl}-D_{kl})\,ds,
\end{split}
\end{align}
where the symbol $\del_{ij}$ represents a partial derivative with respect to the $(i,j)-$component of the underlying space of $n\times n-$matrices. Thus, having at hand
\[
\del_{ij}\Phi(A)=(|A|^2+\delta^2)^{\frac{q-2}{2}} A_{ij}
\]
\[ \del_{kl}\del_{ij}\Phi(A)=(|A|^2+\delta^2)^{\frac{q-2}{2}}\delta_{ij}^{kl}+(q-2)(|A|^2+\delta^2)^{\frac{q-4}{2}}A_{kl}A_{ij}
\]
we deduce from \eqref{diff-CD} the following
\begin{align*}
\underset{ij}{\sum}&\big((|C_{ij}|^2+\delta^2)^{\frac{q-2}{2}}C_{ij}-(|D_{ij}|^2+\delta^2)^{\frac{q-2}{2}}D_{ij}\big)(C_{ij}-D_{ij})\\
&=\underset{ijkl}{\sum}\int_0^1\Big[ \big( |D+s(C-D)|^2 +\delta^2 \big)^{\frac{q-2}{2}}|C_{ij}-D_{ij}|^2 \\
&\qquad +(q-2)\big( |D+s(C-D)|^2 +\delta^2 \big)^{\frac{q-4}{2}}(D_{ij}+s(C_{ij}-D_{ij})(D_{kl}+s(C_{kl}-D_{kl}) (C_{ij}-D_{ij})(C_{kl}-D_{kl})\Big]\,ds\\
&=\underset{ijkl}{\sum}\int_0^1\Big[ (|D+s(C-D)|^2 +\delta^2)^{\frac{q-4}{2}}\big((|D+s(C-D)|^2 +\delta^2)|C_{ij}-D_{ij}|^2 \\
&\qquad+(q-2)(D_{ij}+s(C_{ij}-D_{ij})(D_{kl}+s(C_{kl}-D_{kl}) (C_{ij}-D_{ij})(C_{kl}-D_{kl})\big)\Big]\,ds\\
&\geq \underset{ijkl}{\sum}\int_0^1\Big[ (|D+s(C-D)|^2 +\delta^2)^{\frac{q-4}{2}}\Big(\min(1,(q-1))|D+s(C-D)|^2  |C_{ij}-D_{ij}|^2+\delta^2 |C_{ij}-D_{ij}|^2\Big) \Big]\,ds.
\end{align*}
Finally, we conclude that
\[
  \intO \Ss_\delta-\bar{\Ss}_\delta: \nabla  z \,dx \geq \mu \intO |\nabla z|^2\,dx +(\lambda +\mu)\intO |\divv z|^2\,dx
\]
Next, using H\"{o}lder and Sobolev inequalities, we deduce from  \eqref{diffe} and the above estimate
\begin{align*}
\dfrac{1}{2}\Dt \|\sqrt{\rho}& z\|^2_{L^2(\Omega)}+
\mu \|\nabla z \|_{L^2(\Omega)}^2 \\\
&\leq \|\vartheta\|_{L^{3/2}(\Omega)}\|h\|_{L^6(\Omega)}\|z\|_{L^6(\Omega)}+\|\nabla \bar{\uu}\|_{L^{\infty}(\Omega)}\|\sqrt{\rho}z\|_{L^2(\Omega)}^2 +\|\pi \|_{L^2(\Omega)}\|\nabla z\|_{L^2(\Omega)}\\
&\leq \dfrac{\mu}{2}\|\nabla z\|_{L^2(\Omega)}^2+C \|h\|_{L^6(\Omega)}^2\|\vartheta\|_{L^{3/2}(\Omega)}^2+ \|\nabla \bar{\uu}\|_{L^\infty}\|\sqrt{\rho}z\|_{L^2(\Omega)}^2+C\|\pi\|_{L^2(\Omega)}^2.
\end{align*}
Hence, we conclude with
\begin{align}\label{eq46}
\Dt \|\sqrt{\rho}z\|_{L^2}^2+\mu \|\nabla z\|_{L^2}^2\leq M(t) \big( \|\sqrt{\rho}z\|_{L^2}^2+\|\vartheta\|_{L^{3/2}}^2+\|\pi\|_{L^2}^2\big) 
\end{align}
for some non-negative function $M(t)\in L^1(0,T).$ Meanwhile, using mass conservation equation, we can write
\begin{equation*}
\vartheta_t +\nabla \vartheta \cdot u +\nabla \bar{\rho}\cdot z+\vartheta \divv \bar{u} +\rho \divv z=0,
\end{equation*}
and then, we can deduce that
\begin{align*}
|\vartheta_t|+\nabla |\vartheta|\cdot z+\nabla |\vartheta|\cdot \bar{\uu}\leq |\nabla \bar{\rho}||z|+|\vartheta||\divv \bar{\uu}|+\rho |\divv z|.
\end{align*}
Multiplying the above inequality by $|\vartheta|^{1/2}$ and integrate over $\Omega$, we get
\begin{align*}
\Dt \intO |\vartheta|^{3/2}\, dx&\leq C \intO \big[\big(|\vartheta|+\rho\big)|\vartheta|^{1/2}|\divv z|+|\vartheta|^{3/2}|\divv \bar{\uu}|+|\vartheta|^{1/2}|\nabla \bar{\rho}||z|\big]\,dx\\
&\leq C \big(\|\vartheta\|_{L^6}+\|\rho\|_{L^6}+\|\nabla \bar{\rho}\|_{L^2}\big)\|\vartheta\|_{L^{3/2}}^{1/2}\|\nabla z\|_{L^2}+\|\nabla \bar{\uu}\|_{L^{\infty}}\|\vartheta\|_{L^{3/2}}^{3/2}
\end{align*}
whence
\begin{align}\label{eq47}
\begin{split}
\Dt \|\vartheta\|_{L^{3/2}}^{2}&=\dfrac{4}{3} \|\vartheta\|_{L^{3/2}}^{1/2}\Dt \|\vartheta\|_{L^{3/2}}^{3/2}\\
& \leq C\big(\|\vartheta\|_{L^6}+\|\rho\|_{L^6}^2+\|\nabla \bar{\rho}\|_{L^2}^2\big)\|\vartheta\|_{L^{3/2}}^{2}+\dfrac{\mu}{4}\|\nabla z\|_{L^2}^2 +\|\nabla \bar{\uu}\|_{L^\infty}\|\vartheta\|_{L^{3/2}}^{2}\\
& \leq \dfrac{\mu}{4} \|\nabla z\|_{L^2}^2 +N(t) \|\vartheta\|_{L^{3/2}}^{2}
\end{split}
\end{align}
for some  nonnegative $N(t)\in L^1(0,T)$. Finally, remark that we have
\begin{equation}
\pi_t +\nabla \pi\cdot \uu+\nabla \bar{p} \cdot z+\gamma \pi\divv \bar{\uu}+\gamma p\divv z=0.
\end{equation}
Multiplying the above equality by $\pi$ and doing the same treatment as before, we infer with
\begin{align}\label{eq48}
\begin{split}
\Dt \|\pi\|_{L^2}^2&\leq C\big( \|\pi\|_{L^{\infty}(\Omega)}+\|p\|_{L^{\infty}(\Omega)}+\|\nabla \bar{p}\|_{L^{3}(\Omega)}\big)\|\pi\|_{L^2(\Omega)}\|\nabla z\|_{L^2(\Omega)}+C\|\nabla \bar{\uu}\|_{L^{\infty}(\Omega)} \|\pi\|_{L^2(\Omega}^2\\
&\leq \dfrac{\mu}{4}\|\nabla z\|_{L^2}^2 +K(t) \|\pi\|_{L^2}^2
\end{split}
\end{align}
for some nonnegative $K(t)\in L^1(0,T).$
 Combining all estimates \eqref{eq46} -\eqref{eq47}-\eqref{eq48}, we conclude with
\begin{align*}
\dfrac{\mu}{2}\|\nabla z\|_{L^2}^2 &+\Dt \big( \|\sqrt{\rho} z\|_{L^2}^2 +\|\vartheta\|_{L^{3/2}}^2 +\|\pi\|_{L^2}^2\big)\\
&\leq \big (M(t)+N(t)+K(t)\big)\big( \|\sqrt{\rho} z\|_{L^2}^2 +\|\vartheta\|_{L^{3/2}}^2 +\|\pi\|_{L^2}^2\big).
\end{align*}
By applying a Gronwall's argument, we finish the proof of uniqueness of solution.
\begin{remark}
Following the same arguments developed firstly in \cite{Des-SNS} and later in \cite{Kim-SNS}, we can prove a more general uniqueness result by assuming less regularity on $(\rho, \uu)$ and $(\bar{\rho}, \bar{\uu})$.
\end{remark}
\section{Well-posedness of Bingham system: Proof of Theorem \ref{theo-2}}\label{Sec-theo2}
In this section, we prove Theorem \ref{theo-2}, which is devoted to show the local existence of a unique strong solution to system \eqref{Bing-sys}-\eqref{S-bing}.  In the first subsection, we show the construction of solutions based on our result established in the last section. In the last subsection, we prove the uniqueness issue.
\subsection{Construction of solutions}
In this proof, we shall use our result established in Section \ref{Sec-The1}. For $(\rho_0,\uu_0)$ as in Theorem \ref{theo-2}, we consider the regularizing problem on the initial function $\uu_0(x)$
\begin{align}\label{com-bing-cons}
\begin{split}
-\del_x\big( \Ss_0^\delta\big)&=\sqrt{\rho_0} g-a\del_x\rho_0^\gamma,\vspace*{0.1cm}\\
\Ss_0^\delta:&= \mu \del_x \uu_0^\delta +\tau^*\dfrac{\del_x \uu_0^\delta}{\sqrt{|\del_x \uu_0^\delta |^2 +\delta^2}},
\end{split}
\end{align}
with $\delta\leq 1$ and $g\in L^6(\Omega)$. Again such $\uu_0^\delta$ solution to the above equation surely exists, and, moreover, due to the first statement in Proposition \ref{prop-reg} we have $\uu_0^\delta \in W^{2,6}(\Omega)$ uniformly with respect to $\delta$.

Next, for $(\rho_0, \uu_0^\delta)$ such that $\rho_0$ satisfies the hypothesis of Theorem \ref{theo-2}, and $\uu_0^\delta$ in $W^{2,6}$ solution of \eqref{com-bing-cons}, we know that there exists $(\rho_\delta, \uu_\delta)$ solution to the following problem
\begin{equation}\label{Reg-Bing}
\begin{array}{ccc}
\del_{t}\rho_\delta+\del_x(\rho_\delta \uu_\delta)=0\vspace*{0.2cm}\\
\del_{t}(\rho_\delta \uu_\delta)+\del_x(\rho_\delta \uu_\delta^2)-\del_x\big(\mu \del_x \uu_\delta +\tau^*(|\del_x \uu_\delta |^2 +\delta^2)^{\frac{-1}{2}}\del_x \uu_\delta\big)  +\del_x p_\delta=\rho_\delta f
\end{array}
\end{equation}
with initial data $(\rho_0, \uu_0^\delta)$ provided that the compatibility condition \eqref{com-bing-cons} holds. This is an immediate application of Theorem \ref{theo1} in one-dimensional space and for $q=1$ in \eqref{Sdelta-PL}. Notice that the essential point here is that the $W^{2,p}-$estimates on the elliptic equation hold uniformly with respect to $\delta$. This means, as indicated before, that all the estimates shown on $(\rho_\delta, \uu_\delta)$ in the last section hold uniformly with respect to $\delta$. Accordingly, we have
\begin{equation}\label{est-bing}
\begin{split}
&\underset{0\leq t\leq T^*}{\sup} \big(\|\sqrt{\rho}\del_t \uu_\delta \|_{L^{2}(\Omega)}+\|\del_x \uu_\delta\|_{H^{1}(\Omega)}+\|p_\delta\|_{W^{1,6}(\Omega)}\big)\\
&\quad +\int_0^{T^*} \big(\|\uu_\delta\|_{W^{2,6}(\Omega)}+\|\del_t \del_x \uu_\delta \|_{L^{2}(\Omega)}^{2}+\|\del_t \uu_\delta \|_{L^{6}(\Omega)}^{2}\big)\,dt\leq C,
\end{split}
\end{equation}
where $C$ is a constant that depend on the initial data, however, it does not depend on $\delta$. Now, we move to prove that, up to a subsequence, the solution $(\rho_\delta, \uu_\delta)$ of system \eqref{Reg-Bing}  with initial data $(\rho_0,\uu_0^\delta)$ converges to a solution denoted $(\rho,\uu)$ of \eqref{Bing-sys}-\eqref{S-bing} as $\delta$ tends to zero with initial data $(\rho_0, \uu_0)$. Repeatedly, the main task here is to study the passage to the limit in the viscous stress tensor. Precisely, we want to verify that, at the limit, the pair $(\rho_0, \uu_0)$ satisfies \eqref{S0-comp}. Indeed, for the first part of \eqref{S0-comp}, we remark that the function
$$j(s)=\mu s +\tau^*\dfrac{s}{\sqrt{s^2+\delta^2}},$$
is continuous in $\RR^*$, the norms $\|\Ss_0^\delta\|_{L^p(\Omega)}$ are uniformly bounded in $\delta$, and so by Vitali's Theorem, we get $\Ss_0=j(\del_x \uu)$ for $\del_x \uu(x)\neq 0.$ 

Let us show the second part of \eqref{S0-comp}. As in \cite{Ba-Sh-strong1d}, we argue by contradiction. Set
\[
A= (0,T)\times \Omega \cap \{ \del_x \uu=0\}\cap \{|\Ss_0|\geq \tau^*\},
\]
and assume that ${\rm meas}(A)=m>0$. Then $\psi=\sign (\Ss_0) \cdot\chi_A \in L^{\infty}((0,T)\times \Omega)$ and 
\[
\intT\intO \Ss_0 \psi \,dx\,dt> \tau^* m.
\]
Thus, given a small $\varepsilon>0,$ we have
\[
I^\delta =\intT\intO \Ss_0^\delta \psi \,dx\,dt\geq \Ss_0 m+\varepsilon
\]
for small $\delta$. On the other hand, by the choice of $j$, the inequality $|j(s)|\leq \Ss_0 +\varepsilon/3m$ is valid for small $\delta$ and $|s|<s_0,$ where $s_0$ is small enough. We estimate the integral $I^\delta$. We have
\[
\intT\intO \Ss_0^\delta \psi \,dx\,dt=\intT\int_{|\del_x \uu^\delta|<s_0}\ldots +\intT\int_{|\del_x \uu^\delta|\geq s_0}\ldots=I_1^\delta+I_2^\delta.
\]
By H\"older inequality, we have
\[
|I_2^\delta|\leq \| \Ss_0^\delta\|_{L^2((0,T)\times \Omega)}\; {\rm meas} \big( A\cap \{x; |\del_x\uu^\delta|\geq s_0 \}\big)^{1/2}.
\]
Due to estimate \eqref{est-bing}, we deduce that $|I_2^\delta|\leq \varepsilon/3$ for some small $\delta$. Since $|I_2^\delta|\leq \Ss_0 m+\varepsilon/3$, we arrive at a contradiction.  The proof of Theorem \ref{theo-2} is now finished excluding the uniqueness assertion which will be shown below.

\subsection{Uniqueness of solution}
The proof of uniqueness of solutions to Bingham system \eqref{Bing-sys}-\eqref{S-bing} follows the same lines as in the previous proof of uniqueness of solutions to Power Law system \eqref{Power-Law}-\eqref{Sdelta-PL}. Indeed, we need only to verify that 
\begin{equation}\label{uniq-est}
\intO \Ss-\bar{\Ss}:\del_x z\,dx\geq c\; \|\del_x z\|_{L^2(\Omega)}^2,
\end{equation}
where $\Ss$ is defined in \eqref{S-bing}. Indeed, splitting the domain $\Omega$ into four subdomains
\begin{align*}
&\Omega_1:=\{x\in \Omega \quad {\rm s.t.}\quad \del_x \uu\neq 0 \quad \del_x \bar{\uu}\neq 0\}\quad \quad \Omega_2:=\{x\in \Omega \quad {\rm s.t.}\quad \del_x\uu\neq 0 \quad \del_x\bar{\uu}= 0\}\\
&\Omega_3:=\{x\in \Omega \quad {\rm s.t.}\quad \del_x\uu= 0 \quad \del_x\bar{\uu}\neq 0\}\quad \quad \Omega_4:=\{x\in \Omega \quad {\rm s.t.}\quad \del_x \uu= 0 \quad \del_x\bar{\uu}= 0\},
\end{align*}
we can check easily that \eqref{uniq-est}  holds. This finishes the proof.
\section{Appendix}\label{Appen-sec}
This section is devoted to perform the estimates of each integral term $I_k$ presented in \eqref{26ch}, Section \ref{Sec-The1}. Indeed, 
using Sobolev, H\"older and Young inequalities, we estimate
\begin{multline*}
\begin{split}
&\bullet I_1 \leq 2\|\rho\|_{L^{\infty}(\Omega)}^{1/2}\|\uu\|_{L^{6}(\Omega)}\|\sqrt{\rho}\uu_t\|_{L^3(\Omega)}\|\nabla\uu_t\|_{L^{2}(\Omega)}\\
 &\hspace*{0.65cm} \leq 2\|\rho\|_{L^{\infty}(\Omega)}^{1/2}\|\uu\|_{L^{6}(\Omega)} \|\sqrt{\rho}\uu_t\|_{L^{2}(\Omega)}^{1/2}\|\sqrt{\rho}\uu_t\|_{L^{6}(\Omega)}^{1/2}\|\nabla\uu_t\|_{L^{2}(\Omega)}\\
&\hspace*{0.65cm} \leq C |\rho\|_{L^{\infty}(\Omega)}^{3/4}\|\nabla \uu\|_{L^{2}(\Omega)} \|\sqrt{\rho}\uu_t\|_{L^{2}(\Omega)}^{1/2}\|\nabla\uu_t\|_{L^{2}(\Omega)}^{3/2}\leq C \psi^6+\dfrac{\mu}{16}\|\nabla \uu\|_{L^2(\Omega)}^2,
\end{split}
\end{multline*}
\begin{multline*}
\begin{split}
&\bullet I_2\leq  \|\rho\|_{L^{\infty}(\Omega)}\|\uu\|_{L^6(\Omega)} \|\uu_t\|_{L^6(\Omega)}\|\nabla \uu\|_{L^{3}(\Omega)}^2\\
&\hspace*{0.65cm} \leq C \|\rho\|_{L^{\infty}(\Omega)}\|\nabla \uu\|_{L^2(\Omega)} \|\nabla \uu_t\|_{L^2(\Omega)}\|\nabla \uu\|_{L^{2}(\Omega)}\|\nabla \uu\|_{H^{1}(\Omega)}\leq C\psi^{11}+\dfrac{\mu}{16}\|\nabla \uu_t\|_{L^{2}(\Omega)}^2,
\end{split}
\end{multline*}
\begin{multline*}
\begin{split}
&\bullet I_3\leq \|\rho\|_{L^{\infty}(\Omega)}\|\uu\|_{L^6(\Omega)}^2\|\uu_t\|_{L^6(\Omega)} \|\nabla^2 \uu\|_{L^2(\Omega)}\\
&\hspace*{0.65cm} \leq C \|\rho\|_{L^{\infty}(\Omega)}\|\nabla \uu\|_{L^2(\Omega)}^2 \|\nabla \uu_t\|_{L^2(\Omega)}\|\nabla \uu\|_{H^1(\Omega)}\leq C \psi^{11} +\dfrac{\mu}{16}\|\nabla \uu_t\|_{L^2(\Omega)}^2,
\end{split}
\end{multline*}
\begin{multline*}
\begin{split}
 &\bullet I_4\leq \|\rho\|_{L^{\infty}(\Omega)}\|\uu\|_{L^6(\Omega)}^2 \|\nabla \uu\|_{L^6(\Omega)}\|\nabla \uu_t\|_{L^2(\Omega)}\\
&\hspace*{0.65cm} \leq C \|\rho\|_{L^{\infty}(\Omega)}\|\nabla \uu\|_{L^2(\Omega)}^2 \|\nabla \uu\|_{H^1(\Omega)}\|\nabla \uu_t\|_{L^2(\Omega)}\leq C \psi^{11}+\dfrac{\mu}{16}\|\nabla \uu_t\|_{L^2(\Omega)}^2,
\end{split}
\end{multline*}
\begin{multline*}
\begin{split}
&\bullet I_5\leq \|\rho\|_{L^{\infty}(\Omega)}^{1/2}\|\uu_t\|_{L^6(\Omega)}\|\sqrt{\rho}\uu_t\|_{L^3(\Omega)} \|\nabla \uu\|_{L^2(\Omega)}\\
&\hspace*{0.65cm} \leq C \|\rho\|_{L^{\infty}(\Omega)}^{1/2}\| \uu_t\|_{L^6(\Omega)}\|\sqrt{\rho}\uu_t\|_{L^2(\Omega)}^{1/2} \|\sqrt{\rho}\uu_t\|_{L^6(\Omega)}^{1/2}\|\nabla \uu\|_{L^2(\Omega)}\\
&\hspace*{0.65cm} \leq C \|\rho\|_{L^{\infty}(\Omega)}^{3/4}\| \nabla\uu\|_{L^2(\Omega)}\|\sqrt{\rho}\uu_t\|_{L^2(\Omega)}^{1/2} \|\nabla \uu_t\|_{L^2(\Omega)}^{3/2}
\leq C \psi^{6}+\dfrac{\mu}{16}\|\nabla \uu_t\|_{L^2(\Omega)}^2
\end{split}
\end{multline*}
\begin{multline*}
\begin{split}
&\bullet I_6\leq \|\nabla p\|_{L^3(\Omega)}\|\uu\|_{L^6(\Omega)}\|\nabla \uu_t\|_{L^2(\Omega)}\leq C \psi^3 +\dfrac{\mu}{16}\|\nabla \uu_t\|_{L^2(\Omega)}^2
\end{split}
\end{multline*}
\begin{multline*}
\begin{split}
&\bullet I_7\leq \gamma \| p\|_{L^\infty(\Omega)}\|\uu\|_{L^6(\Omega)}\|\nabla \uu\|_{L^3(\Omega)} \|\nabla^2 \uu\|_{L^2(\Omega)}\\
&\hspace*{0.65cm} \leq C \|p\|_{L^\infty(\Omega)} \|\nabla \uu\|_{L^2(\Omega)}^{3/2}\|\nabla \uu\|_{H^1(\Omega)}^{3/2}\leq C \psi^7.
\end{split}
\end{multline*}
\begin{multline*}
\begin{split}
&\bullet I_8\leq \gamma^2 \|p\|_{L^\infty(\Omega)}\|\nabla\uu\|_{L^3(\Omega)}^3\\
&\hspace*{0.65cm} \leq C \|p\|_{L^\infty(\Omega)}\|\nabla \uu\|_{L^2(\Omega)}^{3/2}\|\nabla \uu\|_{H^1(\Omega)}^{3/2}\leq C\psi^7
\end{split}
\end{multline*}
\begin{multline*}
\begin{split}
&\bullet I_9\leq \|\rho\|_{L^\infty(\Omega)}^{1/2}\|\uu\|_{L^6(\Omega)}\|\sqrt{\rho}\uu_t\|_{L^3(\Omega)}\|\nabla f\|_{L^2(\Omega)}\\
&\hspace*{0.65cm}\leq \|\rho\|_{L^\infty(\Omega)}^{1/2}\|\uu\|_{L^6(\Omega)}\|\sqrt{\rho}\uu_t\|_{L^2(\Omega)}^{1/2}\|\sqrt{\rho}\uu_t\|_{L^6(\Omega)}^{1/2}\|\nabla f\|_{L^2(\Omega)}\\
&\hspace*{0.65cm} \leq C \|\rho\|_{L^\infty(\Omega)}^{3/4}\|\nabla \uu\|_{L^2(\Omega)}\|\sqrt{\rho}\uu_t\|_{L^2(\Omega)}^{1/2}\|\nabla\uu_t\|_{L^2(\Omega)}^{1/2}\|\nabla f\|_{L^2(\Omega)}\leq C\psi^6 +C\|\nabla f\|_{L^2(\Omega)}^2 +\dfrac{\mu}{16}\|\nabla \uu_t\|_{L^2(\Omega)}^2\vspace*{0.2cm}
\end{split}
\end{multline*}
\begin{multline*}
\begin{split}
&\bullet I_{10}\leq \|\rho\|_{L^\infty(\Omega)}\|\uu\|_{L^6(\Omega)}\|f\|_{L^3(\Omega)}\|\nabla \uu_t\|_{L^2(\Omega)}\\
&\hspace*{0.7cm} \leq C \|\rho\|_{L^\infty(\Omega)}\|\nabla \uu\|_{L^2(\Omega)}\|f\|_{L^2(\Omega)}^{1/2}\|f\|_{H^1(\Omega)}^{1/2}\|\nabla \uu_t\|_{L^2(\Omega)} \leq C\psi^6 +C \|f\|_{H^1(\Omega)}^2 +\dfrac{\mu}{16}\|\nabla \uu_t\|_{L^2(\Omega)}^2
\end{split}
\end{multline*}
\begin{multline*}
\begin{split}
&\bullet I_{11}\leq \|\rho\|_{L^{\infty}(\Omega)}^{1/2}\|\sqrt{\rho}\uu_t\|_{L^2(\Omega)}\|f_t\|_{L^2(\Omega)}\leq C \psi^2 +C \|f_t\|_{L^2(\Omega)}^2.
\end{split}
\end{multline*}
\medskip
\hfill\break
{\bf Acknowledgment.} This work was suggested to me by Didier Bresch for him I express my sincere gratitude. The author would like also to thanks Matthias Hieber for various insightful discussions about maximal regularity theory for nonlinear parabolic equations.
\bibliographystyle{abbrv}
\bibliography{Main_Bingham}

\end{document}